\documentclass[12pt]{amsart}
\usepackage{amsmath,amssymb,amsfonts,amsthm,latexsym,graphicx,multirow,enumerate}
\allowdisplaybreaks
\usepackage{tikz}
\usepackage[utf8]{inputenc}
\usepackage[backref]{hyperref}

\usetikzlibrary{calc}
  \newcommand\Cay{\mathrm{Cay}}

\newcommand\K{\mathrm{K}}

\newtheorem{theorem}{Theorem}[section]
\newtheorem{lemma}[theorem]{Lemma}

\theoremstyle{definition}

\usepackage{color}

\definecolor{Blue}{rgb}{0,0,1}
\definecolor{Red}{rgb}{1,0,0}
\definecolor{DarkGreen}{rgb}{0,0.6,0}
\definecolor{DarkYellow}{rgb}{1,1,0.2}
\definecolor{DarkPurple}{rgb}{.6,0,1}

\usepackage{xcolor}
\usepackage[normalem]{ulem}

\begin{document}
	
	\title{ Classification of two-distance-transitive Cayley graphs of  the semi-dihedral groups}

\thanks{Supported by  NSFC (12271524,12331013) and NSF of Hunan (2026JJ50358) and Scientific Research Project of  the Education Department of Hunan Province (24A0142)}

\author[W. Jin]{Wei Jin}
 \address{Wei Jin\\School of Mathematics and Computational Science, Key Laboratory of Intelligent Computing and Information Processing of Ministry of Education\\
Xiangtan University\\
Xiangtan, Hunan, 411105, P.R.China}
\email{jinweipei82@163.com}
\author[C. X. Li]{Cai Xia  Li}
\address{Cai Xia  Li (Corresponding author.)\\School of Mathematics and Computational Science, Key Laboratory of Intelligent Computing and Information Processing of Ministry of Education\\
Xiangtan University\\
Xiangtan, Hunan, 411105, P.R.China}
\email{caixial@yeah.net}
\author[P. S. Li]{Ping Shan Li}
\address{Ping Shan Li\\School of Mathematics and Computational Science, Key Laboratory of Intelligent Computing and Information Processing of Ministry of Education\\
Xiangtan University\\
Xiangtan, Hunan, 411105, P.R.China}
\email{lips@xtu.edu.cn}

 	\begin{abstract}

The class of 2-distance-transitive graphs naturally generalizes distance-transitive graphs and plays a central role in algebraic graph theory. Classifying such graphs for a prescribed underlying group is a key open problem.
A vertex-transitive graph $\Gamma$ is said to be $2$-distance-transitive if, for each $i\in \{1,2\}$, any two  pairs of vertices  with identical distance $i$ in $\Gamma$ can be mapped to each other via some  automorphism of the graph. In this paper, we present a complete classification of  all  $2$-distance-transitive Cayley graphs of the  semi-dihedral groups.
		
	\end{abstract}

	\maketitle
	\section{Introduction}

In this paper, all graphs considered are finite, simple, connected and undirected. For a graph $\Gamma$, we denote by   $V(\Gamma)$, $A(\Gamma)$ and $\operatorname{Aut}(\Gamma)$ its vertex set, arc set and automorphism group, respectively. An \textit{arc} of a graph is  an ordered pair of adjacent vertices. Let $u$ and $v$ be two distinct vertices of $\Gamma$. The \textit{distance} from $u$ to $v$, denoted by $d_{\Gamma}(u,v)$, is the smallest positive integer $n$ for which there exists a path of length $n$ between $u$ and $v$. The \textit{diameter} $\operatorname{diam}(\Gamma)$ of $\Gamma$ is the maximum distance between any two vertices in $V(\Gamma)$.

Let $\Gamma$ be a graph and  $G \le \operatorname{Aut}(\Gamma)$. If $G$ acts transitively on $V(\Gamma)$ or $A(\Gamma)$, then $\Gamma$ is called \textit{$G$-vertex-transitive} or \textit{$G$-arc-transitive}, respectively.
A vertex triple $(u,v,w)$  in which  $v$ is adjacent to both $u$
and $w$ is called a \emph{$2$-arc} if $u\neq w$. We say that $\Gamma$ is   \emph{$(G,2)$-arc-transitive} if, for each $i\leq 2$, the group
$G$ is transitive on the set of  $i$-arcs of $\Gamma$.
In particular,  $\Gamma$ is simply said to be \textit{arc-transitive} or \textit{$2$-arc-transitive},  if it is $\operatorname{Aut}(\Gamma)$-arc-transitive or $(\operatorname{Aut}(\Gamma),2)$-arc-transitive, respectively.
The study of  2-arc-transitive graphs traces its origins to Tutte's foundational work \cite{Tutte-1,Tutte-2}. Since then, this graph family has been extensively investigated, see for instance  \cite{DuMarusicWaller1998,Du2008,Jin2023,LiPan2008,LiSeressSong2015}.

A non-complete $G$-arc-transitive graph $\Gamma$ is said to be \textit{$(G,2)$-distance-transitive} if for any two distinct vertex pairs $(u_1,v_1)$ and $(u_2,v_2)$ satisfying $d_{\Gamma}(u_1,v_1) = d_{\Gamma}(u_2,v_2) = 2$, some element of $G$ maps $(u_1,v_1)$ to $(u_2,v_2)$. When $G = \operatorname{Aut}(\Gamma)$,  $\Gamma$ is called  \textit{$2$-distance-transitive}. We say that $\Gamma$ is \textit{$G$-distance-transitive} if for every integer $i \leq \operatorname{diam}(\Gamma)$, $G$ is transitive on the set of ordered vertex pairs at distance $i$. Accordingly, a graph is \textit{distance-transitive} if it is $\operatorname{Aut}(\Gamma)$-distance-transitive.

By definition, every non-complete $2$-arc-transitive graph is $2$-distance-transitive, but the converse is not necessarily true. If a $2$-distance-transitive graph has girth $3$ (the length of its shortest cycle is $3$), then the graph is not $2$-arc-transitive.  Thus the family of $2$-distance-transitive graphs contains both the family of distance-transitive graphs and the family of $2$-arc-transitive graphs.

Research on (locally) $2$-distance-transitive graphs has developed rapidly in recent years. Devillers, Giudici, Li and Praeger \cite{DevillersGiudiciLiPraeger2012} investigated locally $s$-distance-transitive graphs via the normal quotient method originally proposed for $s$-arc-transitive graphs in \cite{Praeger1993}. Corr, Schneider and the first author \cite{CorrJinSchneider2017} considered $2$-distance-transitive graphs of girth $4$, and classified all such graphs with valency at most $5$ that are $2$-distance-transitive but not $2$-arc-transitive. Later, $2$-distance-transitive Cayley graphs on cyclic groups (circulant graphs) were fully determined   in \cite{ChenJinLi2019}. Most recently, the complete classification of $2$-distance-transitive Cayley graphs over dihedral groups was achieved in \cite{HuangFengZhouYin2025, JinTan2022}.

A \textit{semi-dihedral group} $\mathrm{SD}_{8m}$ of order $8m$ is a non-abelian  group  defined by the following presentation:
\[
\mathrm{SD}_{8m} = \langle a, b \mid a^{4m} = b^2 = 1,\ a^b = a^{2m-1} \rangle,\ m \geq 3.
\]

The semi-dihedral group \(\mathrm{SD}_{8m}\) occupies a distinctive and increasingly significant position at the interface between group theory and algebraic graph theory, especially in the study of Cayley graphs. Its importance arises both from its elegant algebraic structure and from its capacity to generate a wealth of highly symmetric graphs.

The classification program for 2-distance-transitive Cayley graphs was initiated with cyclic groups (circulants) and subsequently extended to dihedral groups. The semi-dihedral group represents the immediate and natural successor in this hierarchical progression.

In this paper, we classify $2$-distance-transitive Cayley graphs on semi-dihedral groups. Our main result is the following:


\begin{theorem}\label{thm1.1}
Let $\Gamma$ be a connected $2$-distance-transitive Cayley graph on the semi-dihedral group $T = SD_{8m} = \langle a, b \mid a^{4m} = 1, \, b^2 = 1, \, a^b = a^{2m-1} \rangle$, where $m \geq 3$ and \( T  \leq \text{Aut}(\Gamma) \). Then $\Gamma$ is isomorphic to one of the following graphs:
\begin{itemize}
    \item [(1)] $\K_{x[y]}$ for some  $x\geq 3$ and $y\geq 2$ with $xy = 8m$;
    \item [(2)] \( \K_{4m,4m} \) where \( m \geq 3 \);
    \item [(3)] \( \K_{4m,4m} - 4mK_2 \) where \( m \geq 3 \);
    \item [(4)] \( B(\operatorname{PG}(d-1,q)) \) and \( B'(\operatorname{PG}(d-1,q)) \), where \( d \geq 3 \), \( q \) is a prime power with $\frac{q^{d}-1}{q-1} = 4m$;
    \item [(5)] \(X_1(4, q)\) where \(q \equiv 3 \pmod{4}\) and \(q = 2m - 1\);
    \item [(6)] \( \K_{q+1}^{2d} \) for some \( d \geq 2 \) dividing \( q-1 \), where \( q \) is an odd prime power, with $d(1+q) = 4m$ and  \( q+1 \ge 6\);
    \item [(7)] \( \Gamma(d,q,r) \), where \( d \geq 2 \), \( q \) is a prime power and \( r \mid q-1 \) with $r\frac{q^d-1}{q-1}  = 4m$;

\end{itemize}
	  \end{theorem}

The graphs described in Theorem \ref{thm1.1} will be defined in the next section.

    \bigskip

This paper is organized as follows. After the introduction, Section~$2$ collects essential definitions from group and graph theory, as well as several elementary lemmas used throughout the paper. Theorem~\ref{thm1.1},   proved in Section 3, presents the  classification result for connected \(2\)-distance-transitive Cayley graphs of the   semi-dihedral groups.
The main method of this paper is based on Praeger's normal subgroup reduction method \cite{Praeger1993a,Praeger1993,Praeger1997}, in which the quotienting is always performed with respect to a normal subgroup. Therefore, to complete the classification carried out in this paper, we first investigate nontrivial normal subgroups of the semi-dihedral group \(T = \mathrm{SD}_{8m}\) and the structures of their quotient groups \(T/N\), this is done in Lemmas~\ref{normal group} and ~\ref{T/N}.
Let \(\Gamma\) be a connected \(2\)-distance-transitive Cayley graph on the semi-dihedral group
$T = \mathrm{SD}_{8m} = \langle a, b \mid a^{4m} = 1,\; b^2 = 1,\; a^b = a^{2m-1} \rangle$,
where \(T \leq G= \operatorname{Aut}(\Gamma)\) and \(m \geq 3\).
Assume further that \(\Gamma \ncong \K_{x[y]}\) for all \(x \ge 3\) and \(y \ge 2\).
Let \(N\) be a maximal normal subgroup of \(G\) having at least three orbits on \(V(\Gamma)\).
By Theorem~\ref{(G/N, s)-distance-transitive}, \(\Gamma\) is a cyclic or metacyclic cover of the quotient graph \(\Gamma_N\).
Moreover, \(\Gamma_N\) is either a complete graph or a non-complete \((G/N,2)\)-distance-transitive graph, and \(G/N\) acts quasiprimitively or bi-quasiprimitively on \(V(\Gamma_N)\).
Based on these facts, we classify all such base graphs and explicitly construct their cyclic and metacyclic covers.

\bigskip
\bigskip

	\section{Preliminaries}
	\label{sec2}

    In this section, we will give some definitions about groups and graphs that will be used
    in the paper. And we also present some relevant lemmas. For the group theoretic terminology not defined here we refer the reader to \cite{Cameron1999,Wielandt1964}.
    \medskip

\subsection{Groups and related results}

 Let $G$ be a permutation group on a set $\Omega$ and take $\alpha \in \Omega$. The \emph{stabilizer} of $\alpha$ in $G$, denoted $G_\alpha$, is the subgroup of $G$ consisting of all elements that fix $\alpha$. We denote by $G_\Delta$ and $G_{\{\Delta\}}$ the pointwise-stabilizer and setwise-stabilizer of $\Delta \subseteq \Omega$, respectively. We say that $G$ is \emph{semiregular} on $\Omega$ if $G_\alpha = 1$ for every $\alpha \in \Omega$, and \emph{regular} if it is both transitive and semiregular.

Let $\mathbb{Z}_n$ denote the cyclic group of order $n$.

A \textit{dihedral group} of order $2n$ is denoted by $D_{2n}$ and defined as $D_{2n} = \langle a, b \mid a^n = 1,\ b^2 = 1,\ bab = a^{-1} \rangle$.

A group $G$ is called \emph{metacyclic} if there is a normal subgroup $N$ of $G$ such that both $N$ and $G/N$ are cyclic.

Let $G$ be a transitive permutation group on a set $\Omega$ and let $\mathcal{B}$ be a $G$-invariant partition of $\Omega$. If the only possibilities for $\mathcal{B}$ are the partition into one part, or the partition into singletons then $G$ is called \textit{primitive}. The \textit{kernel} of $G$ on $\mathcal{B}$ is the normal subgroup of $G$ consisting of all elements that fix setwise each $B \in \mathcal{B}$. We call $\mathcal{B}$ \textit{maximal} if $G^{\mathcal{B}}$ is primitive on $\mathcal{B}$. The group $G$ is said to be \textit{quasiprimitive}, if every non-trivial normal subgroup of $G$ is transitive on $\Omega$, while $G$ is said to be \textit{bi-quasiprimitive} if every non-trivial normal subgroup of $G$ has at most two orbits on $\Omega$ and there exists one which has exactly two orbits on $\Omega$. Quasiprimitivity is a natural generalization of primitivity: while every normal subgroup of a primitive group is transitive, there exist quasiprimitive groups that are not primitive. For further background on quasiprimitive and bi-quasiprimitive permutation groups, we refer the reader to \cite{LiPanXia2021,Praeger1993}.

Given two groups $H$ and $L$, we write $H \times L$ for their direct product, $H:L$ for a semidirect product of $H$ by $L$, and $H.L$ for an extension of $H$ by $L$. For a group $G$, let $Z(G)$ denote its centre and $\operatorname{soc}(G)$ its socle, namely the product of all minimal normal subgroups of $G$.

A finite   abstract group $H$ is called a \textit{$B$-group}, if every permutation group $G$ containing  $H $ as a regular subgroup of $\Omega$  is either imprimitive or $2$-transitive on $\Omega$. It is known that all finite cyclic groups with composite order and all dihedral groups are  $B$-groups \cite{Wielandt1964}, and  the class of   $B$-groups also includes all dicyclic group   \cite{Wielandt1964}.

The following   lemma further establishes that  certain families of   abelian groups belong to the class of   $B$-groups.

\begin{lemma} [{\cite[Theorem 25.5]{Wielandt1964}}]\label{thm:25.5}
If \( H \) is an abelian group  of type \( (p^a, p^b) \) with \( a > b \), then \( H \) is a \( B \)-group.
\end{lemma}

\begin{lemma}\label{B-group}
Let \( G \) be a quasiprimitive permutation group on \( \Omega \) admitting   a regular subgroup $T$ of prime-power degree. Then \( G \) is primitive on $\Omega$. In particular,
if \(T = SD_{2^{n+2}}\), with   $n >2$, then \( G \) is $2$-transitive on $\Omega$.

\end{lemma}
\begin{proof}
By \cite[Theorem 2.2]{LiPanMa2009}, every   quasiprimitive permutation group of
prime-power degree is primitive. Note that  the semi-dihedral group \( SD_{2^{n+2}}\) has order $2^{n+2}$, so   $|\Omega| = 2^{n+2}$ is a prime-power. This implies that   \( G \) is primitive on \( \Omega \). Furthermore,
since $G$ is a  prime-power degree permutation group  containing the  regular semi-dihedral subgroup $T$,
it follows from   a theorem of Kanazawa and Enomoto    \cite[Theorem]{KanazawaEnomoto1968} that    \( G \) is a $B$-group. Consequently, the primitivity of  \( G \) forces $G$ to be   $2$-transitive on $\Omega$. \end{proof}

The following result lists all subgroups of the group $\mathrm{PSL}(2, p^f)$ where $p$ is a prime and $f$ is a positive integer.

\begin{theorem}[{\cite{Hup67}}]\label{Dickson, 8.27}
Let  $p$ be  a prime and and $f$  a positive integer. Then the group \( \mathrm{PSL}(2, p^f) \) has only the following subgroups:
\begin{itemize}
    \item [(1)] Elementary abelian $p$-groups;
    \item [(2)] Cyclic groups of order $z$, where $z$ is a divisor of $\dfrac{p^f+1}{k}$ and $k = (p^f - 1, 2)$;
    \item [(3)] Dihedral groups of order $2z$, where $z$ is as defined in (2);
    \item [(4)] The alternating group $A_4$; this can occur only for $p > 2$ or when $p = 2$ and $f \equiv 0 \pmod{2}$;
    \item [(5)] The symmetric group $S_4$; which occurs if and  only if $p^{2f} \equiv 1 \pmod{16}$;
    \item [(6)] The alternating group $A_5$ for $p = 5$ or $p^{2f} \equiv 1 \pmod{5}$;
    \item [(7)] A semidirect product of an elementary abelian group of order $p^m$ with a cyclic group of order $t$, where $t$ is a divisor of $(p^m - 1,\; p^f - 1)$;
    \item [(8)] The group $\operatorname{PSL}(2, p^m)$ for $m$ a divisor of $f$, or the group $\operatorname{PGL}(2, p^m)$ for $2m$ a divisor of $f$.
\end{itemize}
\end{theorem}

To further characterize the higher transitivity of permutation groups, we recall the following classical classification result concerning $3$-transitive permutation groups.

\begin{lemma}[{\cite{Cameron1999}}]\label{3-trans}
Let $G$ be a $3$-transitive permutation group of degree at least $4$. Then one of the following occurs.
\begin{itemize}
    \item [(1)]The socle of $G$ is $3$-transitive;
    \item [(2)]$\operatorname{PSL}(2,q) \trianglelefteq G \leq \operatorname{P\Gamma L}(2,q)$ with natural action on the projective line $\operatorname{PG}(1,q)$ of degree $q+1$, for odd $q \geq 5$, and with the socle of $G$ being isomorphic to $\operatorname{PSL}(2,q)$ and not $3$-transitive;
    \item [(3)]$G = \operatorname{AGL}(m,2)$, $m \geq 3$;
    \item [(4)]$G = \mathbb{Z}_2^4 : A_7 < \operatorname{AGL}(4,2)$;
    \item [(5)]$G = \operatorname{PGL}(2,3) \cong \operatorname{AGL}(2,2) \cong S_4$ of degree $4$.
\end{itemize}
\end{lemma}

\bigskip
\subsection{Graphs and related results}
We now fix the graph notation used throughout this paper.
Let $\mathrm{C}_n$ denote   the cycle graph of length $n$, $\mathrm{K}_n$ the complete graph of order $n$, and $\mathrm{K}_{n,n}$ the complete bipartite graph of order $2n$. We write  $\mathrm{K}_{n,n} - n\mathrm{K}_2$ for  the graph obtained by deleting a perfect matching from $\mathrm{K}_{n,n}$, and $\mathrm{K}_{x[y]}$ for the complete multipartite graph with $x\ge 3$ parts each of size $y\ge 2$.

Let $T$ be a finite group, and let $S\subseteq T$ satisfy $1\notin S$ and $S=S^{-1}$. The \textit{Cayley graph} $\Cay(T,S)$ is defined as the graph with vertex set $T$ and edge set $\{\{g,sg\}\mid g\in T,\,s\in S\}$. In particular, $\Cay(T,S)$ is connected if and only if $T=\langle S\rangle$. Let $R(T)=\{\sigma_t\mid t\in T\}$ denote the group of right translations of $T$, where $\sigma_t:x\mapsto xt$. Then $R(T)$ is a regular subgroup of $\operatorname{Aut}(\Gamma)$, and we routinely identify $T$ with $R(T)$. As established by Godsil \cite[Lemma 2.1]{Godsil1981}, $N_{\operatorname{Aut}(\Gamma)}(T)=T:\operatorname{Aut}(T,S)$, where $\operatorname{Aut}(T,S)=\{\sigma\in\operatorname{Aut}(T)\mid S^\sigma=S\}$ denotes the automorphism group of the pair $(T,S)$. A Cayley graph $\Gamma$ is called \textit{normal} if $\operatorname{Aut}(\Gamma)=N_{\operatorname{Aut}(\Gamma)}(T)$, a notion first introduced by Xu \cite{Xu1998}. It is a classical result that a graph $\Gamma$ is a Cayley graph on a group $G$ if and only if $\operatorname{Aut}(\Gamma)$ contains a regular subgroup isomorphic to $G$, see \cite[Lemma~16.3]{Biggs1993}.

A graph with $2n$ vertices is called   a \textit{bicirculant} if its automorphism group contains an element of order $n$ that  is  fixed-point-free on the vertex set  and decomposes the vertex set into  exactly two disjoint $n$-cycles. The classification of finite $2$-arc-transitive bicirculants has been accomplished   in \cite{Jin2023}. In particular,  every  Cayley graph on  a  semi-dihedral group is a bicirculant.

For each integer $d \geq 3$ and each prime power $q$, let $B(\operatorname{PG}(d-1,q))$ denote  the bipartite graph whose vertices are the $1$-dimensional and $(d-1)$-dimensional subspaces of a $d$-dimensional vector space over $\operatorname{GF}(q)$, where two subspaces are adjacent if and only if one is contained in the other. Denote by $B'(\operatorname{PG}(d-1,q))$ the bipartite complement of $B(\operatorname{PG}(d-1,q))$, that is, the bipartite graph with the same vertex set but where a $1$-subspace and a $(d-1)$-subspace are adjacent if and only if their intersection is the zero subspace. Both graphs have the same automorphism group $P\Gamma L(d,q) : \mathbb{Z}_2$, see \cite{Du2008}.

Let $X$ be a connected graph and let $N$ be a finite group. Suppose $\psi : \operatorname{Arc}(X) \to N$ is a function assigning a group element to each arc of $X$ such that $\psi(v,u) = (\psi(u,v))^{-1}$ for every arc $(u,v) \in \operatorname{Arc}(X)$. Then $\psi$ is called a \textit{voltage assignment}, the values of $\psi$ are called \textit{voltages}, and $N$ is the \textit{covering group}. In particular, $\psi$ is called \textit{reduced} if the values of $\psi$ on a spanning tree are trivial (equal to the identity element of $N$). From this we may construct a larger graph $\operatorname{Cov}(X,\psi) = X \times_\psi N$, called the \textit{(derived) voltage graph} (or \textit{cover graph}), with vertex set $V(X) \times N$ and adjacency defined by $(u,g) \sim (v,h)$ if and only if $u \sim v$ and $h = g\psi(u,v)$.

\smallskip
Below we introduce some voltage graphs:
\begin{itemize}
\item [(1)] The  \emph{graph $X_1(4, q)$}:  Let $q$ be a prime power such that $q \equiv 3 \pmod{4}$ and let $F(q) = \operatorname{GF}(q)$, $S(q)$ be the set of all non-zero squares of $F(q)$ and $N(q)$ be the set of non-zero non-squares of $F(q)$. Let $Y = K_{q+1}$, and identify the vertices of $Y$ with the projective line $\operatorname{PG}(1, q) = F(q) \cup \{\infty\}$ (see \cite[pp.~283, 285]{DuMarusicWaller1998}). We define $X_1(4, q)$ to be the $4$-fold cover $\operatorname{Cov}(Y, \psi)$, where the voltage $\psi : \operatorname{Arc}(Y) \to \mathbb{Z}_4$ is defined with the following rule:
\[
\psi(x, y) :=
\begin{cases}
0, & \text{if } \infty \in \{x, y\}; \\
1, & \text{if } y - x \in S(q); \\
3, & \text{if } y - x \in N(q).
\end{cases}
\]

\item [(2)] The  \emph{graph $X_2(3)$}:  Let $\Gamma = K_{5,5} - 5K_2$, where $V(\Gamma) = \{1,2,3,4,5\} \cup \{1',2',3',4',5'\}$,
\[
E(\Gamma) = \{ ij' \mid i \neq j,\ i,j' \in V(\Gamma) \}.
\]

Define $X_2(3) = \Gamma \times_\psi \mathbb{Z}_3$, with the voltage assignment $\psi : \operatorname{Arc}(\Gamma) \to \mathbb{Z}_3$ such that
\[
\begin{aligned}
\psi_{1,2'} &= \psi_{1,3'} = \psi_{1,4'} = \psi_{1,5'} = \psi_{2,1'} = \psi_{2,3'} = \psi_{3,1'} \\
&= \psi_{3,2'} = \psi_{4,1'} = \psi_{4,5'} = \psi_{5,1'} = \psi_{5,4'} = 0, \\[6pt]
\psi_{2,5'} &= \psi_{3,4'} = \psi_{4,3'} = \psi_{5,2'} = 1, \\[6pt]
\psi_{2,4'} &= \psi_{3,5'} = \psi_{4,2'} = \psi_{5,3'} = 2.
\end{aligned}
\]

\item [(3)] The  \emph{graph $AT_D(4, 6)$}:  We define $N$-covers of $K_4$ with covering transformation groups $N = \langle a, b \rangle \cong D_6$, where $V(K_4) = \{1, 2, 3, 4\}$.

 $AT_D(4, 6) = K_4 \times_\psi D_6$, with the voltage assignment $\psi : \operatorname{Arc}(K_4) \to D_6$ defined by
\[
\begin{aligned}
\psi_{1,2} &= b, & \psi_{1,3} &= ba, & \psi_{1,4} &= ba^{-1}, \\
\psi_{2,3} &= ba^{-1}, & \psi_{2,4} &= ba, & \psi_{3,4} &= b.
\end{aligned}
\]

\item [(4)] The  \emph{graph $K_{q+1}^{2d}$}: Let $q = r^l$ for an odd prime $r$ and let $\operatorname{GF}(q)^* = \langle \theta \rangle$ be the multiplicative group of the field $\operatorname{GF}(q)$ of order $q$. For any $d \mid q-1$ and $d \ge 2$, define a voltage graph $K_{q+1}^{2d} = (K_{q+1,q+1} - (q+1)K_2) \times_\psi \mathbb{Z}_d$, where
\[
\begin{aligned}
\psi_{\infty', i} &= \psi_{\infty, j'} = \overline{0},  i, j \ne \infty; \\
\psi_{i, j'} &= \overline{h},  j - i = \theta^h,\ i, j \ne \infty,\ h \text{ is an integer.}
\end{aligned}
\]

\item [(5)] The  \emph{graph $X(2,2)$}:  Let $V(2,2)$ be the 2-dimensional vector space over $\mathbb{F}_2$. Let $K_{4,4}$ be the complete bipartite graph with the bipartition $V(K_{4,4}) = U \cup W$, where $U = \{ \alpha \mid \alpha \in V(2,2) \}$ and $W = \{ \alpha' \mid \alpha \in V(2,2) \}$. Define the cover $X(2,2) = K_{4,4} \times_\psi \mathbb{Z}_2$ with the voltage assignment
\[
\psi(\alpha, \beta') = \alpha \beta^T = a_1b_1 + a_2b_2,
\]
for all $\alpha = (a_1, a_2)$ and $\beta = (b_1, b_2) \in V(2,2)$.

\item [(6)] The  \emph{graph $CQ(k,n)$}:  Let $k$ be an integer such that $1 \le k \le n-1$ and $(k,n)=1$. Then, $k \in \mathbb{Z}_n^*$. Denote by $k^{-1}$ the inverse of $k$ in $\mathbb{Z}_n^*$. Let $V(Q_3) = \{\mathbf{a}, \mathbf{b}, \mathbf{c}, \mathbf{d}, \mathbf{w}, \mathbf{x}, \mathbf{y}, \mathbf{z}\}$ and let $T$ be a spanning tree consisting of the edges $\{\mathbf{ax}, \mathbf{ay}, \mathbf{az}, \mathbf{bw}, \mathbf{bz}, \mathbf{cz}, \mathbf{dy}\}$. Now, for any two non-negative integers $k$ and $n$ with $1 \le k \le n-1$ and $(k,n)=1$, the graph $CQ(k,n)$ is defined to have vertex set $V(CQ(k,n)) = V(Q_3) \times \mathbb{Z}_n$ and edge set
\[
\begin{aligned}
E(CQ(k,n)) = \{ & (\mathbf{a},i)(\mathbf{x},i),\ (\mathbf{a},i)(\mathbf{y},i),\ (\mathbf{a},i)(\mathbf{z},i),\ (\mathbf{b},i)(\mathbf{w},i), \\
& (\mathbf{b},i)(\mathbf{z},i),\ (\mathbf{c},i)(\mathbf{z},i),\ (\mathbf{d},i)(\mathbf{y},i), \\
& (\mathbf{b},i)(\mathbf{y},i+1),\ (\mathbf{c},i)(\mathbf{w},i+k),\ (\mathbf{c},i)(\mathbf{x},i-k^{-1}), \\
& (\mathbf{d},i)(\mathbf{w},i-k^{-1}-1),\ (\mathbf{d},i)(\mathbf{x},i+k) \mid i=0,1,\dots,n-1\},
\end{aligned}
\]
where all numbers $i+t$ ($i,t \in \mathbb{Z}_n$) are taken modulo $n$.
Note that the graph $CQ(1,3) \cong GP(12,5)$ has order $24$, while $CQ(1,6) \cong GP(24,5)$ has order $48$.

\item [(7)] The \emph{graph $\Gamma(d,q,r)$}:  A group divisible design $\mathcal{D}$ with parameters $(n,m;k;\lambda_1, \\ \lambda_2)$, denoted by $\operatorname{GDD}(n,m;k;\lambda_1,\lambda_2)$, is an ordered triple $(\mathcal{P},\mathcal{G},\mathcal{B})$, where $\mathcal{P}$ is a set of points, $\mathcal{G}$ is a partition of $\mathcal{P}$ into $m$ sets of size $n$, each set being called a \textit{group}, and $\mathcal{B}$ is a collection of $k$-subsets, called \textit{blocks} of $\mathcal{P}$ so that each pair of points from the same group occurs in exactly $\lambda_1$ blocks and each pair of points from different groups occurs in exactly $\lambda_2$ blocks. The triple $\mathcal{I} = (\mathcal{P},\mathcal{B},I)$ of a group divisible design is an incidence structure with the natural incidence relation $I$, and the dual incidence structure $\mathcal{I}^* = (\mathcal{B},\mathcal{P},I^*)$. If there exists a partition $\mathcal{G}'$ of $\mathcal{B}$ such that the triple $(\mathcal{B},\mathcal{G}',\mathcal{P})$ is a $\operatorname{GDD}(n,m;k;\lambda_1,\lambda_2)$, then we call that the $\mathcal{D}$ is a \textit{group divisible design with dual property} with parameters $(n,m;k;\lambda_1,\lambda_2)$ and we denote such a design by $\operatorname{GDDDP}(n,m;k;\lambda_1,\lambda_2)$.

Let $d \ge 2$ be an integer and let $q$ be a prime power. Let $V$ be a vector space of dimension $d$ over $\operatorname{GF}(q)$, the finite field with $q$ elements. We define the set of non-zero vectors in $V$ as the point set $\mathcal{P}$ and the set of affine hyperplanes not containing zero in $V$ as the block set $\mathcal{B}$, that is, $\mathcal{P} = \{x \in V \mid x \ne 0\}$ and $\mathcal{B} = \{x + H \mid H \text{ is a hyperplane in } V \text{ and } x \in V \setminus H\}$.

We make a partition $\mathcal{G}$ of $\mathcal{P}$ such that the collinear non-zero vectors in $V$ belong to the same group in $\mathcal{G}$. Note that each group in $\mathcal{G}$ has size $q-1$. Then $(\mathcal{P};\mathcal{G};\mathcal{B})$ is a $\operatorname{GDD}(n,m;k;0;\lambda_2)$, where $n = q-1$, $m = \frac{q^d-1}{q-1}$ is the number of projective points in $V$, $k = q^{d-1}$ is the number of affine hyperplanes containing a given non-zero vector, and $\lambda_2 = q^{d-2}$ is the number of affine hyperplanes containing two given non-zero and non-collinear vectors.

We look at the dual incidence structure $\mathcal{I}^*$ of $\mathcal{I} = (\mathcal{P};\mathcal{B};I)$. We make a partition $\mathcal{G}'$ of $\mathcal{B}$ such that the parallel affine hyperplanes belong to the same group in $\mathcal{G}'$. Then $(\mathcal{B};\mathcal{G}';\mathcal{P})$ becomes a $\operatorname{GDD}(n,m;k;0;\lambda_2)$, where $n = q-1$, $m = \frac{q^d-1}{q-1}$ is the number of $(d-1)$-dimensional subspaces in $V$, $k = q^{d-1}$ is the number of non-zero vectors in an affine hyperplane, and $\lambda_2 = q^{d-2}$ is the number of non-zero vectors in the intersection of two given non-parallel affine hyperplanes.

This shows that $\mathcal{D}(d;q) := (\mathcal{P};\mathcal{G};\mathcal{B})$ is a $\operatorname{GDDDP}(n;m;k;0;\lambda_2)$, where $n,m,k,\lambda_2$ are given above. We denote $\Gamma(d,q) := \Gamma(\mathcal{D}(d;q))$ as the point-block incidence graph of $\mathcal{D}(d;q)$. It is clear that the general linear group $GL(d,q)$ acts as a group of automorphisms of the graph $\Gamma(d,q)$. For any $r \mid q-1$, let $N \leq  Z(GL(d,q)) \cong \mathbb{Z}_{q-1}$ and $|N| = (q-1)/r$. Let $\Gamma(d,q,r)$ be the quotient graph of $\Gamma(d,q)$ induced by $N$. Then $V(\Gamma(d,q,r))=2r\frac{q^d-1}{q-1}$.
\end{itemize}
\bigskip

The following fact is well-known, and can be easily proved.

\begin{lemma}\label{|N|=p}
Let \( X \mapsto Y \) be a regular cyclic covering of a connected graph such that some $2$-distance-transitive group \( G \leq \text{Aut}(X) \) projects to   \( Y \) via the covering group. Then there exists a regular prime cyclic covering \( X' \mapsto Y \) such that some $2$-distance-transitive group \( G' \leq \text{Aut}(X') \) projects along \( X' \mapsto Y \).
\end{lemma}

Let $\Gamma$ be a graph and let $G \leqslant \operatorname{Aut}(\Gamma)$ act transitively on $V(\Gamma)$. Suppose $\mathcal{B} = \{B_1, B_2, \dots, B_n\}$ is a partition of $V(\Gamma)$. The \emph{quotient graph} of $\Gamma$ with respect to $\mathcal{B}$, denoted $\Gamma_{\mathcal{B}}$, is defined as the graph with vertex set $\mathcal{B}$ where $\{B_i, B_j\}$ is an edge if and only if there exist $x \in B_i$ and $y \in B_j$ such that $\{x, y\} \in E(\Gamma)$.

We say that $\Gamma_{\mathcal{B}}$ is \emph{non-trivial} if $1 < |\mathcal{B}| < |V(\Gamma)|$. If for every pair of adjacent blocks $B_i, B_j$ and every $x \in B_i$ we have $|\Gamma(x) \cap B_j| = 1$, then $\Gamma$ is called a \emph{cover} of $\Gamma_{\mathcal{B}}$.

If $\mathcal{B}$ is $G$-invariant, then $G$ induces an action on $\Gamma_{\mathcal{B}}$ as a subgroup of $\operatorname{Aut}(\Gamma_{\mathcal{B}})$. When the blocks of $\mathcal{B}$ are the orbits of a non-trivial normal subgroup $N$ of $G$, we write $\Gamma_{\mathcal{B}} = \Gamma_N$ and refer to it as the \emph{$G$-normal quotient} of $\Gamma$. If, in addition, $\Gamma$ is a cover of $\Gamma_{\mathcal{B}}$, then $\Gamma$ is called a \emph{$G$-normal cover} of $\Gamma_{\mathcal{B}}$. Suppose that $\Gamma$ is a cover of $\Gamma_{\mathcal{B}}$. If the kernel of the induced $G$-action on $\mathcal{B}$ is cyclic (respectively metacyclic), then $\Gamma$ is called a \emph{cyclic cover} (respectively \emph{metacyclic cover}) of $\Gamma_{\mathcal{B}}$.

\begin{lemma}[{\cite[Lemma 2.5]{LiPan2008}}]\label{LiPan2008}
Let \(\Gamma\) be a connected \(G\)-arc-transitive graph. Suppose that \(N\) is a normal subgroup of \(G\) such that \(\Gamma\) is an \(N\)-cover of the normal quotient graph \(\Gamma_N\). Then \(\Gamma\) is \((G,2)\)-arc-transitive if and only if \(\Gamma_N\) is \((G/N,2)\)-arc-transitive.
\end{lemma}

A graph on $n$ vertices is called a \textit{circulant} if it admits an automorphism that is an $n$-cycle. Consequently, every circulant is a Cayley graph on  a cyclic group.

A graph is called a \textit{dihedrant} if it is a Cayley graph on  a dihedral group.
The classification of connected $2$-arc-transitive and $2$-distance-transitive Cayley graphs on dihedral groups are given in the following lemmas.

\begin{lemma}[{\cite[Theorem 1.2]{Du2008}}]\label{Du2008}
Let \( n \geq 3 \) and let \( X \) be a connected $2$-arc-transitive Cayley graph of a dihedral group of order \( 2n \). Then one of the following occurs:
\begin{itemize}
    \item [(1)]{\( X \)} is a basic graph and is isomorphic to one of the following graphs:
    \( C_{2n} \) (where \( n \) is a prime); \( \K_{2n} \); \( \K_{n,n} \); \( B(H_{11}) \) or \( B^{\prime}(H_{11}) \); \( B(\text{PG}(d,q)) \) or \\ \( B^{\prime}(\text{PG}(d,q)) \) (where \( n = \frac{q^d - 1}{q - 1} \), \( d \geq 2 \), and \( q \) is a prime power);
    \item [(2)]$X$ is not a basic graph and either $X$ is isomorphic to $\mathrm{K}_{n,n} - n\mathrm{K}_2$, or there exists an odd prime power $q$ such that $n = q + 1$ and $X$ is isomorphic to $\mathrm{K}_{q+1}^{2d}$, where $d$ is a divisor of $\frac{q-1}{2}$ if $q \equiv 1 \pmod{4}$, and a divisor of $q - 1$ if $q \equiv 3 \pmod{4}$, respectively.
\end{itemize}
\end{lemma}


The following proposition determines all the vertex quasiprimitive arc-transitive bicirculants.

\begin{lemma}[{\cite[Proposition 4.2]{DevillersGiudiciJin2022}}]\label{DevillersGiudiciJin2022}
Let $\Gamma$ be a connected $G$-arc-transitive bicirculant over the cyclic subgroup $H$ of order $n$ such that $G$ is quasiprimitive on $V(\Gamma)$. Then one of the following holds:
\begin{itemize}
    \item [(1)] $G$ is primitive on $V(\Gamma)$ and $\Gamma$ is one of the following graphs:
        \begin{itemize}
            \item[(1.1)] $\K_{2n}$;
            \item[(1.2)] Petersen graph or its complement, and $A_5 \leq G \leq S_5$;
            \item[(1.3)] $H(2, 4)$ or its complement, and $G$ is a rank $3$ subgroup of $\operatorname{AGL}(4, 2)$;
            \item[(1.4)] Clebsch graph or its complement, and $G$ is a rank $3$ subgroup of $\operatorname{AGL}(4, 2)$.
        \end{itemize}
    \item [(2)] $G$ is not primitive on $V(\Gamma)$ and $\Gamma$ is one of the following graphs:
        \begin{itemize}
            \item[(2.1)] $\K_{n[2]}$ and $G$ has rank $3$ on vertices;
            \item[(2.2)] $\K_{n,n} - n\K_2$ with $\operatorname{PGL}(d, q) \leq G \leq \operatorname{P\Gamma L}(d, q)$ and $n = (q^d - 1)/(q - 1)$.
        \end{itemize}
\end{itemize}
\end{lemma}

A graph $\Gamma$ is said to be \textit{locally $(G, s)$-distance-transitive} with $s \geq 1$ if, for every vertex $u \in V(\Gamma)$, the vertex stabilizer $G_u$ acts transitively on the set $\Gamma_i(u)$ for all integers $i \leq s$. The following lemma provides a reduction method for the study of locally $(G, s)$-distance-transitive graphs.

\begin{lemma}[{\cite[Lemma 5.3]{DevillersGiudiciLiPraeger2012}}]\label{DevillersGiudiciLiPraeger2012} Let \(\Gamma\) be a connected locally \((G, s)\)-distance-transitive graph with \(s \geq 2\). Let \(1 \neq N \triangleleft G\) be intransitive on \(V(\Gamma)\), and let \(\mathcal{B}\) be the set of \(N\)-orbits on \(V(\Gamma)\). Then one of the following holds:

\begin{itemize}
\item[(i)] \(|\mathcal{B}| = 2\).

\item[(ii)] \(\Gamma\) is bipartite, \(\Gamma_N \cong K_{1,r}\) with \(r \geq 2\) and \(G\) is intransitive on \(V(\Gamma)\).

\item[(iii)] \(s = 2\), \(\Gamma \cong K_{m[b]}\), \(\Gamma_N \cong K_m\) with \(m \geq 3\) and \(b \geq 2\).

\item[(iv)] \(N\) is semiregular on \(V(\Gamma)\), \(\Gamma\) is a cover of \(\Gamma_N\), \(|V(\Gamma_N)| < |V(\Gamma)|\) and \(\Gamma_N\) is \((G/N, s')\)-distance-transitive where \(s' = \min\{s, \operatorname{diam}(\Gamma_N)\}\).
\end{itemize}
\end{lemma}



\begin{lemma}
    The graph \( B(\operatorname{PG}(d-1,q)) \), where \( d \geq 3 \), \( q \) is a prime power with $\frac{q^{d}-1}{q-1} = 4m$ is not a  Cayley graph of the group ${SD}_{8m} = \langle a, b \mid a^{4m} = b^2 = 1,\ a^b = a^{2m-1} \rangle,\ m \geq 3$.
\end{lemma}
\begin{proof}
  $B(\operatorname{PG}(d-1,q))$ is the bipartite graph with vertices the $1$-dimensional and $(d-1)$-dimensional subspaces of a $d$-dimensional vector space over $\operatorname{GF}(q)$, and two subspaces are adjacent if and only if one is contained in the other. Then by \cite[p. 210]{ChengOxley1987}, the graph has the automorphism group $\operatorname{P\Gamma L}(d,q) : \mathbb{Z}_2$. We next prove that ${SD}_{8m}$ cannot be embedded in $\operatorname{P\Gamma L}(d,q) : \mathbb{Z}_2$.

  We proceed by contradiction. Assume that ${SD}_{8m} \leq \operatorname{P\Gamma L}(d,q) : \mathbb{Z}_2$, then $a \in \operatorname{P\Gamma L}(d,q)$ ia a Singer cycle of order $4m$. If $b \notin \operatorname{P\Gamma L}(d,q)$, then $a^b = a^{-1}$. This is a contradiction. We next consider the case $b \in \operatorname{P\Gamma L}(d,q)$.
 By \cite[p. 187, Satz 7.3]{Hup67}, the normalizer \(N_{\operatorname{P\Gamma L}(d,q)}(\langle a\rangle)\) takes the form \(\langle a\rangle:\langle \delta\rangle\), where \(\langle \delta\rangle\cong\mathbb{Z}_d\) and \(a^\delta=a^q\). As \(a^b=a^{2m-1}\),  we have \(\langle a,b\rangle\le \langle a,\delta\rangle\), thus there exists an integer \(r\) with \(0 < r < d\) such that \(a^{\delta^r}=a^{q^r}=a^b=a^{2m-1}\), which yields the congruence \(q^r\equiv 2m-1\pmod{4m}\). Furthermore, substituting \(4m=(q^d-1)/(q-1)\) gives \(2q^r = q^{d-1}+q^{d-2}+\cdots+q-1\). A straightforward computation shows that this cannot occur. Thus The graph \( B(\operatorname{PG}(d-1,q)) \) can not a  Cayley graph of the group ${SD}_{8m}$. \end{proof}

\begin{lemma}
    Let $X$ be the graph \( \Gamma(d,q,r) \) with order $8m$, where \( d \geq 2 \) is an even integer, \( q \) is a prime power and \( r \mid q-1 \). If there exists an integer $k$ satisfying $0<k<d$ such that $d\mid 2k$ and $q^k \equiv 2m+1 \pmod{4m}$, then $X$  is a  Cayley graph of the group ${SD}_{8m} = \langle a, b \mid a^{4m} = b^2 = 1,\ a^b = a^{2m-1} \rangle,\ m \geq 3$.
\end{lemma}
\begin{proof}

Since $\Gamma(d,q)$ is a point-block incidence graph, the vertices of the graph $\Gamma(d,q)$ are divided into two types, that is a set $\mathcal{P}$ of points and a set $\mathcal{B}$ of blocks. Note that the point-block incidence graph of an incidence structure is a bipartite graph. We denote that $\mathcal{P} = \mathbb{F}_{q^d} \setminus \{0\}$ and  $\mathcal{B} = \{H_x \mid x \in \mathbb{F}_{q^d} \setminus \{0\}\}$, each block is uniquely determined by a nonzero normal vector $x$, defined as $H_x = \{z \mid z^t x = 1\}$.
Thus, every vertex of  the graph  $\Gamma(d,q)$ can be uniformly written as an ordered pair: $V = \{(\epsilon,x) \mid \epsilon \in \{P,B\},\ x \in \mathbb{F}_{q^d}^*\}$, where $(P,x)$ stands for the point $x$ and $(B,x)$ stands for the block $H_x$. Then $|\Gamma(d,q)| = (q^d - 1) + (q^d - 1) = 2(q^d - 1)$.

Define the map $\phi$ swaps points and blocks while keeping the underlying field element labels unchanged, that is $\phi(P,x) = (B,x)$ and $\phi(B,x) = (P,x)$. Thus we have $\phi^2 = 1$.

Let $q$ be a power of $p$. We shall choose the finite field $\mathbb{F}_{q^d}$ as our $d$-dimensional vector space over $\mathbb{F}_q$ on which $GL(d,q)$ acts. Let $\alpha$ be a generator of the multiplicative group $\mathbb{F}_{q^d}\setminus\{0\}$ and define $\tau:\mathbb{F}_{q^d}\to\mathbb{F}_{q^d}$ by $\tau(z)=\alpha z$. Then $\tau$ generates a cyclic subgroup $\langle \tau \rangle$ of order $q^d-1$ in $GL(d,q)$, known as a Singer cycle.

Let $\sigma$ be the automorphism of $\mathbb{F}_{q^d}$ defined by $\sigma(z)=z^q$ for all $z\in\mathbb{F}_{q^d}$. Then $\sigma$ generates a cyclic subgroup $\langle \sigma \rangle$ of order $d$. Since $\sigma$ acts $\mathbb{F}_q$-linearly on $\mathbb{F}_{q^d}$ and $\sigma^{-1}\tau\sigma=\tau^q$, it follows that $\sigma$ normalizes $\langle \tau \rangle$.


The automorphism $\sigma$ induces an automorphism of the graph, whose action is described as follows:
the action on points is $\sigma(P,x)=(P,x^q)$. Since $\sigma$ is a field automorphism, it maps the hyperplane $H_x=\{z\mid z^t x=1\}$ to $H_{x^q}=\{w\mid w^t x^q=1\}$. Thus the action on blocks is $\sigma(B,x)=(B,x^q)$.
That is, $\sigma$ preserves the type label entirely and only replaces the field element $x$ with its $q$ power.

Since $\phi$ only acts on the "type coordinate" while $\sigma$ only acts on the "field element coordinate", they are independent and commute with respect to the direct product structure. Hence, for any integer exponent $k$, we have $\phi \sigma^k = \sigma^k \phi$. Moreover, It is easy to see that $\phi ^{-1}\tau\phi = \tau^{-1}$.

Note that $\Gamma(d,q,r)$ is the quotient graph of $\Gamma(d,q)$ induced by $N$, where $N \leq Z = Z(GL(d,q)) \cong \mathbb{Z}_{q-1}$, and for some $r \mid q-1$, $|N| = (q-1)/r$. By the definition of graph $\Gamma(d,q,r)$, we have  $|\Gamma(d,q,r)| = \frac{2(q^d - 1)}{|N|} = \frac{2r(q^d - 1)}{q - 1} =8m$. For notational convenience, let $\overline{X} = \Gamma(d,q)$ and $X = \overline{X}_N = \Gamma(d,q,r)$. Since the general linear group $GL(d, q)$ acts as a group of automorphism
of the graph $\Gamma(d,q)$, $GL(d,q) \leq \mathrm{Aut}(\overline{X})$ and $GL(d,q)/N \leq \mathrm{Aut}(X)$.

Let $a$ denote the image of $\tau$ in the quotient group, that is, $a = \tau N$.
Since the order of $\tau$ is $q^d-1$ and $|N| = (q-1)/r$, it follows that the order of $a$ is $4m$.
Let $b = \phi \sigma^k N$. Using $\phi \sigma^k = \sigma^k \phi$ and $d\mid 2k$ we obtain
\[
b^2 = (\phi \sigma^k N)(\phi \sigma^k N) = \phi^2 \sigma^{2k} N = \sigma^{2k} N = N,
\]
hence $b$ has order $2$. Moreover, from $q^k \equiv 2m+1 \pmod{4m}$ and $\sigma^{-1}\tau\sigma = \tau^{q}$ we deduce
\[
\begin{aligned}
b^{-1}ab &= \sigma^{-k} \phi N \,\tau N\, \phi \sigma^k N
          = \sigma^{-k} (\phi \tau \phi) \sigma^k N
          = \sigma^{-k} \tau^{-1} \sigma^k N \\
         &= (\sigma^{-k} \tau \sigma^k)^{-1} N
          = (\tau^{q^k})^{-1} N
          = \tau^{2m-1} N
          = a^{2m-1}.
\end{aligned}
\]

Since the group $\langle \tau, \phi \rangle$ is transitive on $V(\overline{X})$, we have the group $\langle a, b \rangle$ is transitive on $V(X)$. Let $T = \langle a, b \rangle \leq \mathrm{Aut}(X)$, then  the order of $T$ is $8m$  and $T \cong {SD}_{8m} = \langle a, b \mid a^{4m} = b^2 = 1,\ a^b = a^{2m-1} \rangle,\ m \geq 3$. Thus $X$  is a  Cayley graph of the group ${SD}_{8m}$.
\end{proof}

   \medskip




    \bigskip

\section{Proof of Theorem \ref{thm1.1}}

In this section, we will prove our main  theorem by a series of lemmas.

The first lemma lists all normal subgroups of semi-dihedral groups.

\begin{lemma}\label{normal group}
Let $T = \mathrm{SD}_{8m} = \langle a, b \mid a^{4m} = 1,\, b^2 = 1,\, bab = a^{2m-1} \rangle$ with $m \geq 3$. Then all nontrivial proper normal subgroups $N$ of $T $ are precisely:
\begin{itemize}
    \item [(1)]   $N = \langle a^i \rangle$, where    $i$ is a divisor of $4m$ with $1\leq i<4m$;
    \item [(2)]   $N = \langle a^2, ab \rangle \cong Q_{4m}$;
    \item [(3)]  $N = \langle a^2, b \rangle \cong D_{4m}$;
       \item [(4)]  $m$ is odd,  $N = \langle a^4, b \rangle$ or $ \langle a^4, a^2b \rangle$, and $N\cong D_{2m}$.
\end{itemize}
\end{lemma}

\begin{proof}
The group  $T$ decomposes as the disjoint union  $T = SD_{8m} = \langle a \rangle \cup \langle a \rangle b$, so its   elements can be partitioned into two  subsets of equal size as follows:
\[
A = \{a^i \mid 0 \leq i \leq 4m - 1\} \quad \text{and} \quad B = \{a^i b \mid 0 \leq i \leq 4m - 1\}.
\]

By the formula for the order of a group element, we have \( o(a^i) = \frac{4m}{\gcd(i, 4m)} \).
	
For an element $a^ib$ in $B$, we compute its square using the relation   $(a^i b)^2 = a^i b a^i b = a^i \cdot a^{i(2m-1)} = a^{i + i(2m-1)} = a^{2mi}$. Thus   $o(a^i b) = 2$ when $i$ is even and $o(a^i b) = 4$ when $i$ is odd.
In particular,   every  element of \( B \) has order \( 2 \) or \( 4 \).

Suppose first that    $N \leq \langle a \rangle$. Then $N = \langle a^i \rangle$ for some    $ i \mid 4m $. Since $|T:\langle a \rangle |=2$, $\langle a \rangle $ is a normal subgroup of $T$.   Moreover,   every subgroup of a cyclic group is characteristic, so  $N = \langle a^i \rangle$ is a normal subgroup of $T$.

Now assume    \( N \nleq \langle a \rangle \).  Since    \( N \trianglelefteq T \), the intersection  \( N \cap \langle a \rangle \) is normal in  \( \langle a \rangle \), so \( N \cap \langle a \rangle = \langle a^d \rangle \) for some divisor $d$ of  \(  4m \), where   $a^d$ is an element of minimal positive exponent in  $A\cap N$.   Because   \( N \nleq \langle a \rangle \), $N$ contains some  element $a^rb $ in  $B$.


By the second isomorphism theorem,
\[
N / (N \cap \langle a \rangle) \cong N \langle a \rangle / \langle a \rangle\leq T / \langle a \rangle.
\]
Since  \(T / \langle a \rangle\) has order 2, the quotient   $N / (N \cap \langle a \rangle)$ has order either 1 or 2.
Order 1 would imply  $N = N \cap \langle a \rangle \leq \langle a \rangle$,  contradicting our assumption. Hence
$N / (N \cap \langle a \rangle)$ has order  2, and so $N=\langle a^d \rangle\cup \langle a^d \rangle a^rb$. Thus,  \( N = \langle a^d, a^rb \rangle \).

Because   \( N = \langle a^d, a^rb \rangle \trianglelefteq T \), it is closed under conjugation by the generators $a$ and $b$  of $T$.
Conjugating $a^d$ by either generator yields an element of $\langle a^d \rangle$, which is automatically in $N$. We now impose normality on the generator $a^r b$.


Conjugate $a^r b$ by $a$:
\[
(a^r b)^a = a^{-1} a^r b a = a^{r-1} b a = a^{r-1} a^{2m-1} b = a^{2m-2} \cdot a^r b.
\]
Since $(a^r b)^a \in N$ and $a^r b \in N$, we must have $a^{2m-2} \in N$, so $d \mid 2m-2$.

Next, conjugate $a^r b$ by $b$:
\[
(a^r b)^b = b^{-1} a^r b b = a^{r(2m-1)} b = a^{2r(m-1)} \cdot a^r b.
\]
Since $(a^r b)^b \in N$ and $a^r b \in N$, it follows that $a^{2r(m-1)} \in N$, so $d \mid 2r(m-1)$. Combined with $d \mid 2(m-1)$, this reduces to $d \mid 2r$.

Since $d \mid 2m-2$ and $d \mid 4m$, and since $\gcd(2m-2,4m) \in \{2,4\}$, we have
$d \in \{2,4\}$.
As $N$ is a nontrivial normal subgroup,  $d \geq 2$. Consequently,
\[
\begin{cases}
d = 2 \text{ or } 4, & \text{if } m \text{ is odd}; \\
d = 2, & \text{if } m \text{ is even}.
\end{cases}
\]

If  $d = 2$, then   $N = \langle a^2, a^rb \rangle$. When  $r$ is odd, $r-1$ is even, and  $ab=(a^2)^{\frac{4m-r+1}{2}}(a^rb)$,
so we may replace $a^rb$ by $ab$, consequently   $N = \langle a^2, ab \rangle \cong Q_{4m}$. When $r$ is even, $a^r\in \langle a^2 \rangle$, so   $N = \langle a^2, b \rangle \cong D_{4m}$.

If  $d = 4$ (which forces   $m$ odd),  then $N = \langle a^4, a^rb \rangle$. Since $d$ divides $2r$, $r$ is necessarily even, which yields $N = \langle a^4, a^2b \rangle$ or $\langle a^4, b \rangle$. We conclude that $N$ is isomorphic to the dihedral group $D_{2m}$ of order $2m$.

Conversely, it is straightforward to verify that each subgroup listed above is normal in $T$: the subgroups $\langle a^i \rangle$ are characteristic in the normal subgroup $\langle a \rangle$; the subgroups $\langle a^2, ab \rangle$ and $\langle a^2, b \rangle$ have index $2$; and the remaining subgroups, when $m$ is odd, are invariant under conjugation by $a$ and $b$ by the relations derived above. Hence the list is complete.
\end{proof}

\begin{lemma}\label{T/N}
Let $T = \mathrm{SD}_{8m} = \langle a, b \mid a^{4m} = 1,\, b^2 = 1,\, bab = a^{2m-1} \rangle$ with $m \geq 3$. Suppose $N$ is a normal subgroup of $T$ such that $N = \langle a^i \rangle$, where $i$ is a proper divisor of $4m$ and $i \neq 1, 4m$. Then one of the following four cases holds:
\begin{itemize}
    \item [(1)] $T/N \cong \mathbb{Z}_2 \times \mathbb{Z}_2$ if $i = 2$;
    \item [(2)] $T/N \cong D_{2i}$ if $i \geq 3$ and $i \mid 2m$;
    \item [(3)] $T/N \cong \mathbb{Z}_4 \times \mathbb{Z}_2$ if $i = 4$ and $i \nmid 2m$;
    \item [(4)] $T/N \cong \mathrm{SD}_{2i}$ if $i \geq 3$, $i \neq 4$ and $i \nmid 2m$.
\end{itemize}
\end{lemma}

\begin{proof}

Since $N \trianglelefteq T$, every element of $T/N$ can be expressed as $a^{i'} b^{j'} N$, where $i'$ is a positive integer with $1 \le i' \le 4m-1$ and $j' \in \{0,1\}$. Hence $T/N$ is generated by $aN$ and $bN$. Set $x = aN$ and $y = bN$. From $bab = a^{2m-1}$, we obtain $babN = a^{2m-1}N = (a^{2m}N)(a^{-1}N)$ and $yxy = x^{2m-1}$.

First suppose $i \mid 2m$. Since $N = \langle a^i \rangle$, we have $a^{2m} \in N$ and $yxy = x^{-1}$.

If $i = 2$, then $N = \langle a^2 \rangle$. Note that $x^k = (aN)^k = a^k N = N = \overline{1}$, which implies $a^k \in N = \langle a^2 \rangle$ and thus $k=2$. Similarly, $y^l = (bN)^l = b^l N = N = \overline{1}$ yields $b^l \in N = \langle a^2 \rangle$ and $l=2$. Therefore, both $x = aN$ and $y = bN$ have order $2$. Consequently,
\[
T/N = \langle x, y \mid x^2 = \overline{1},\ y^2 = \overline{1},\ yxy = x^{-1} \rangle \cong \mathbb{Z}_2 \times \mathbb{Z}_2,
\]
so Case (1) is satisfied.

If $i \geq 3$, note that $x^k = (aN)^k = a^k N = N = \overline{1}$, then we have $a^k \in N = \langle a^i \rangle$ and $k=i$. Similarly, $y^l = (bN)^l = b^l N = N = \overline{1}$ implies $l=2$. So the order of $x$ is $i$, and $y$  has order $2$. Thus
\[
T/N = \langle x, y \mid x^i = \overline{1},\ y^2 = \overline{1},\ yxy = x^{-1} \rangle \cong D_{2i},
\]
which gives Case (2).

Next suppose $i \nmid 2m$. As $i$ divides $4m$ but not $2m$, we necessarily have $4 \mid i$.

If $i = 4$, then $N = \langle a^4 \rangle$. Since $x^k = (aN)^k = a^k N = N = \langle a^4 \rangle$, it follows that $a^k \in N = \overline{1}$ and $k=4$. Similarly, $y^l = (bN)^l = b^l N = N = \overline{1}$ yields $b^l \in N = \langle a^2 \rangle$ and $l=2$. Thus $x$ has order $4$ and $y$ has order $2$. Moreover, the condition $i \nmid 2m$ implies that $m$ is odd. Write $m = 2r+1$ for some positive integer $r$. Since
\[
babN = a^{2m}N a^{-1}N = a^{2(2r+1)}N a^{-1}N = a^{4r+1}N = aN,
\]
it follows that $yxy = x$, and therefore
\[
T/N = \langle x, y \mid x^4 = \overline{1},\ y^2 = \overline{1},\ yxy = x \rangle \cong \mathbb{Z}_4 \times \mathbb{Z}_2.
\]
Hence Case (3) holds.

Now consider $i \neq 4$. Then $N = \langle a^i \rangle$. Since $x^k = (aN)^k = a^k N = N = \overline{1}$, we have $k = i$. As $y^l = (bN)^l = b^l N = N = \overline{1}$, it follows that $l=2$. So the order of $x$ is $i$, and $y$ has order $2$.

Recall $4 \mid i$, and set $i = 4d$. Write $m = 2^t f$ where $f$ is odd. Then $4m = 2^{t+2}f$ and $2m = 2^{t+1}f$. Since $i \mid 4m$ and $i \nmid 2m$, we have $i = 2^{t+2}f'$ for some odd divisor $f'$ of $f$. Thus $d = 2^{t}f'$ and $\frac{m}{d} = \frac{f}{f'}$ is odd. Let $\dfrac{m}{d} = 2r+1$. Then
\[
a^{2m} = a^{2d\cdot \frac{m}{d}} = a^{2d(2r+1)} = a^{2d} \cdot a^{4dr} = a^{2d} \cdot a^{ir}.
\]
Since $a^{ir} \in N$, we get
\[
babN = a^{2m-1}N = a^{2m}Na^{-1}N = a^{2d}N \cdot a^{-1}N .
\]
Thus, $yxy = x^{2d-1}$. Consequently,
\[
T/N = \langle x, y \mid x^{4d} = \overline{1},\ y^2 = \overline{1},\ yxy = x^{2d-1} \rangle \cong \mathrm{SD}_{8d} \cong \mathrm{SD}_{2i},
\]
and Case (4) is verified.
\end{proof}

\begin{lemma}\label{quot-notkmb-1}
Let $\Gamma$ be a connected $(G,2)$-distance-transitive  graph.
 Suppose that $\Gamma\ncong \K_{m'[b']}$ for any  $m'\geq 3$ and $b'\geq 2$.
Let $N$ be an intransitive normal subgroup of $G$  such that $\Gamma$ is a cover of $\Gamma_N$. Then  $\Gamma_N\ncong \K_{m[b]}$ for any  $m\geq 3$ and $b\geq 2$.
\end{lemma}

\begin{proof}

Assume to the contrary that $\Gamma_{N}\cong \K_{m[b]}$ for some $m\geq 3$ and $b\geq 2$.
Let $C_1,C_2,\ldots,C_m$ be the corresponding $m$ blocks   of $V(\Gamma_N)$, set $C_i=\{B_{i1},B_{i2},\ldots,B_{ib}\}$ where  each $B_{ij}$ is an   $N$-orbit.
Let $B_{ij}=\{u_{ij1},u_{ij2},\ldots,u_{ij|N|}\} $ with $u_{ijk}\in V(\Gamma)$.

Since $[C_i]\cong \overline{\K_b}$ in $\Gamma_N$, the blocks    $B_{ie}$ and $B_{if}$ are non-adjacent whenever $e\neq f$,
and so no vertex of $B_{ie}$ is  adjacent to any vertex of $B_{if}$.
Moreover,  $B_{ie}$ and $B_{jf}$ are adjacent whenever $i\neq j$. In particular,  $B_{21}$ is adjacent to both $B_{11}$ and $B_{12}$. Because   $\Gamma$ is a cover of $\Gamma_N$,
  the vertex $u_{211}$ is adjacent to a unique  vertex of $B_{11}$, say $u_{111}$, and to a unique vertex of $B_{12}$, say $u_{121}$.
Then   $(u_{111},u_{211},u_{121})$ forms   a $2$-geodesic of $\Gamma$.

Since   $m\geq 3$, there exists a block $C_x$ with    $x\neq 1,2$.
From   the fact   $\Gamma_{N}\cong \K_{m[b]}$ it follows   that   $B_{11}$ is    adjacent to all $N$-orbits of $C_x$, and so $u_{111}$ is adjacent to a vertex of each $B_{xi}$.
Without loss of generality, assume   $u_{111}$ is adjacent to  $u_{x11}\in B_{x1}$.

If $u_{211}$ were  not adjacent to $u_{x11}$, then $(u_{211},u_{111},u_{x11})$ would be   a 2-geodesic of $\Gamma$.
Since $B_{11}$ is adjacent to $B_{22}$ in $\Gamma_N$,
 $u_{111}$ is adjacent to  some   vertex of  $ B_{22}$, say $u_{221}$.
Then $(u_{211},u_{111},u_{221})$ is also   a 2-geodesic of $\Gamma$.
By the  $(G,2)$-distance-transitivity of  $\Gamma$, the vertex stabilizer $G_{u_{211}}$
contains    an element $g$ sending    $u_{221}$ to $u_{x11}$. Consequently,
$g$ maps  $B_{22}$ to $B_{x1}$,
and hence    $C_{2}$ to $C_{x}$. This   is impossible, because
$G_{u_{211}}$ fixes $C_2$ setwise. Thus $u_{211}$ is  adjacent to $u_{x11}$, that is, $u_{x11}\in \Gamma(u_{211})$.

Since $u_{x11}$ is an arbitrary  vertex in $\Gamma(u_{111})\cap (C_3\cup C_4\cup \cdots \cup C_m )$, we obtain
$\Gamma(u_{111})\cap (C_3\cup C_4\cup \cdots \cup C_m )\subseteq \Gamma(u_{111})\cap \Gamma(u_{211}).$

Now  $\Gamma(u_{111})\cap C_2\cap \Gamma(u_{211})=\emptyset$,  and  $\Gamma(u_{111})=[\Gamma(u_{111})\cap C_2]\cup [\Gamma(u_{111})\cap (C_3\cup C_4\cup \cdots \cup C_m )]$, so    $\Gamma(u_{111})\cap \Gamma(u_{211})\subseteq \Gamma(u_{111})\cap (C_3\cup C_4\cup \cdots \cup C_m ) $.
Hence
$$\Gamma(u_{111})\cap (C_3\cup C_4\cup \cdots \cup C_m )= \Gamma(u_{111})\cap \Gamma(u_{211}).$$

Next, suppose  $u_{121}$ were   not adjacent to $u_{x11}$. Then $(u_{121},u_{211},u_{x11})$ would be   a 2-geodesic of $\Gamma$.
Recall   $u_{121}\in \Gamma_2(u_{111})$. By distance-transitivity,        $G_{u_{121}}$
contains    an element $g$  mapping   $u_{111}$ to $u_{x11}$. Thus
$g$ sends  $B_{11}$ to $B_{x1}$,
and hence
  $C_{1}$ to $C_{x}$, contradicting the fact that
$G_{u_{121}}$ fixes $C_1$ setwise. Therefore    $u_{121}$ is  adjacent to $u_{x11}$, that is, $u_{x11}\in \Gamma(u_{121})$.
As   $u_{x11}$ is arbitrary in $\Gamma(u_{111})\cap (C_3\cup C_4\cup \cdots \cup C_m )$, we get
$$\Gamma(u_{111})\cap (C_3\cup C_4\cup \cdots \cup C_m )\subseteq  \Gamma(u_{121}).$$

Now let $w\in \Gamma(u_{111})\cap \Gamma_2(u_{211})$.
Since $u_{221}\in \Gamma_2(u_{211})$ and $G_{u_{211}}$ has an element that moves $w$ to $u_{221}$ while  fixing $C_2$ setwise,   $w$ lies   in some $B_{2i}\in C_2$ with   $i\neq 1$.
A similar argument as above shows   that $\Gamma(u_{111})\cap (C_3\cup C_4\cup \cdots \cup C_m )= \Gamma(u_{111})\cap \Gamma(w)$.
Hence $w$ and $u_{x11}$ are adjacent.

If    $u_{121}$ were   not adjacent to $w$, then $(u_{121},u_{x11},w)$ would be   a 2-geodesic of $\Gamma$.
Then   $G_{u_{121}}$ would contain
  an element sending    $u_{111}$ to $w$, moving
  $C_{1}$ to $C_{2}$, a contradiction. Thus $u_{121}$ is  adjacent to $w$, $w \in \Gamma(u_{121})$  and consequently,   $\Gamma(u_{111})\cap \Gamma_2(u_{211}) \subseteq \Gamma(u_{121})$.
Therefore,  $[\Gamma(u_{111})\cap \Gamma_2(u_{211})]\cup \{u_{211}\}\cup [\Gamma(u_{111})\cap (C_3\cup C_4\cup \cdots \cup C_m )] \subseteq \Gamma(u_{121})$.

From   $\Gamma_{N}\cong \K_{m[b]}$, $\Gamma_N$ has valency $b(m-1)$. Since   $\Gamma$ covers $\Gamma_N$,  $\Gamma$ also has valency $b(m-1)$, with   $|\Gamma(u_{111})\cap C_2|=b$ and  $|\Gamma(u_{111})\cap (C_3\cup C_4\cup \cdots \cup C_m )|=b(m-2)$.
Hence $|\Gamma(u_{111})\cap \Gamma(u_{211})|=b(m-2)$,  $|\Gamma_2(u_{111})\cap \Gamma(u_{211})|=|\Gamma(u_{111})\cap \Gamma_2(u_{211})|=b-1$ and $|\Gamma(u_{121})|=b(m-1)$.

The three  sets $\Gamma(u_{111})\cap \Gamma_2(u_{211})$, $ \{u_{211}\}$ and  $\Gamma(u_{111})\cap (C_3\cup C_4\cup \cdots \cup C_m )$ are pairwise disjoint,
and their union has   $(b-1)+1+b(m-2)=b(m-1)$ elements. Thus this union  equals
 $\Gamma(u_{121})$, and consequently,
$$\Gamma(u_{121})=\Gamma(u_{111}).$$

Since
$u_{121}\in \Gamma_2(u_{111})$ and $\Gamma$ is $(G,2)$-distance-transitive,
we have $\Gamma(z)=\Gamma(u_{111})$  for every vertex $z\in \Gamma_2(u_{111})$.
Therefore $\Gamma$ is antipodal of   diameter 2, and so
$\{u_{111}\}\cup \Gamma_2(u_{111})$ is a nontrivial block of the $G$-action on $V(\Gamma)$. Hence
 $\Gamma\cong \K_{m'[b']}$ for some $m'\geq 3$ and $b'\geq 2$,  contradicting the hypothesis  that $\Gamma$ is not a complete multipartite graph.

Thus   $\Gamma_N\ncong \K_{m[b]}$ for any  $m\geq 3$ and $b\geq 2$.
\end{proof}

\begin{lemma}\label{N < T}
Let $\Gamma $ be a connected  $G$-arc-transitive Cayley graph on  the group  \[
T = \langle a, b \mid a^{4m} = 1,\, b^2 = 1,\, a^b = a^{2m-1} \rangle,
\] where  $T \leq G \leq \operatorname{Aut}(\Gamma)$ and $m \geq 3$.
Let $H=\langle a \rangle$. Let   $N$ be a normal subgroup of $G$ such that $N$ is semiregular on $V(\Gamma)$. Then
one of the following holds:
\begin{itemize}
    \item [(1)]   $N\leq  R(H) < R(T)$ and $\Gamma_N$ is a Cayley graph on  $R(T)/N$;
    \item [(2)]   $ |N:N\cap R(H)|=2$ and $\Gamma_N$ is a circulant on  $R(H) N/N$. Furthermore,
    $G$ contains a normal subgroup $X$ satisfying $|N:X|=2$ or $4$.
    \end{itemize}

\end{lemma}

\begin{proof}
Suppose that $N$ is a normal subgroup of $G$ acting   semiregularly on $V(\Gamma)$.  Then
 the set of $N$-orbits, $\mathcal{B}$, is a $G$-invariant partition of $V(\Gamma) = T$. Hence $R(T)$ induces a group $R(T)N/N$ which is a subgroup of $G/N \leq \operatorname{Aut}(\Gamma_N)$. Since $R(T)$ is regular on $V(\Gamma)$, it follows that $R(T)N/N \cong R(T)/(R(T) \cap N)$ is transitive on $V(\Gamma_N)$.

Since $R(H)$ is a normal subgroup of $R(T)$ with $|R(T):R(H)|=2$, it follows that  $H$ and $Hb$ are  the two orbits of $R(H)$ in $V(\Gamma)$ and
$R(b)$ swaps $H$ and $Hb$.
Take   $B\in \mathcal{B}$. Then $Bb\in \mathcal{B}$.

Assume first that $B\subset H$ or $Hb$. Without loss of generality, assume  $B\subseteq H$. Then $Bb\in Hb$. Set $\mathcal{B}_0:=\{B^{R(h)}\mid h\in H\}$ and $\mathcal{B}_1:=\{(Bb)^{R(h)}\mid h\in H\}$.
Then $H=\bigcup_{h\in H}B^h$ and $Hb=\bigcup_{h\in H}(Bb)^{R(h)}$.

Because $N$ is semiregular, it is  regular on each of its orbits, so   $|B|=|Bb|=|N|$. Consequently, $|\mathcal{B}_i||N|=|Hb|=|H|$ for $i=0,1$. As $R(H)$ acts regularly on both $H$ and $Hb$, its image $R(H)N/N\cong R(H)/(R(H)\cap N)$ acts regularly on  $\mathcal{B}_0$ and $\mathcal{B}_1$. Hence   $|R(H)/(R(H)\cap N)|=|\mathcal{B}_i|$ for $i=0,1$. Combining this with $|\mathcal{B}_i||N|=|H|$ yields $|R(H)\cap N|=|N|$. Therefore  $R(H)\cap N=N$, so $N\leq R(H)<R(T)$ and $N$ is cyclic.
It follows that $R(T)/N$ acts regularly on  $V(\Gamma_N)$, and $\Gamma_N$ is a Cayley graph on  the group $R(T)/N$.
This   establishes  part (1).

Now suppose that $B\cap H\neq \varnothing$ and $B\cap Hb\neq \varnothing$.
In this case, $R(H)N/N\cong R(H)/(R(H)\cap N)$ acts  regularly on $\mathcal{B}$. Consequently,
$R(H)\cap N$ is semiregular on $B$ with two orbits $B\cap H$ and $B\cap Hb$, and in particular, $|B\cap H|=|B\cap Hb|$. Thus  $R(H)\cap N$ is a cyclic index 2 normal subgroup of $N$, and  $\Gamma_N$ is a circulant on $R(H) N/N$.

Define $X=\langle g^2|g\in N\rangle$ and $Y=\langle g^2|g\in N\cap R(H)\rangle$.
Then $X$ is a characteristic subgroup of $N$ and so  normal in  $G$. Because $|N:N\cap R(H)|=2$,
we have $g^2\in N\cap R(H)$ for every   $g\in N$, so
$X\leq N\cap R(H)$. As $N\cap R(H)\leq R(H)$ is cyclic, we obtain $|N\cap R(H):Y|=1$ or 2.
Hence $|N:Y|=|N:N\cap R(H)|\times |N\cap R(H):Y|=2$ or 4. Since
$Y\leq X$, it follows that $|N:X|=2$ or 4, which proves  part (2).
\end{proof}

\begin{theorem}\label{(G/N, s)-distance-transitive}
Let $\Gamma$ be a connected $(G,2)$-distance-transitive but not $(G,2)$-arc-transitive Cayley graph on the semi-dihedral group
\[
T = \mathrm{SD}_{8m} = \langle a, b \mid a^{4m} = 1,\, b^2 = 1,\, a^b = a^{2m-1} \rangle,
\]
with $T \leq G \leq \operatorname{Aut}(\Gamma)$ and $m \geq 3$. Assume  $\Gamma \not\cong \mathrm{K}_{x[y]}$ for all $x \geq 3$ and $y \geq 2$. Let $N$ be a normal subgroup of $G$ maximal subject to having at least three orbits. Then $\Gamma$ is a cover of $\Gamma_N$, and $\Gamma_N$ is either a complete graph or a non-complete $(G/N,2)$-distance-transitive but not $(G/N,2)$-arc-transitive graph. Moreover, $G/N$ acts faithfully, quasiprimitively or bi-quasiprimitively on $V(\Gamma_N)$.

If $\Gamma_N$ is non-complete, then $N = \langle a^i \rangle$ for some divisor $i$ of $4m$ with $i \neq 1, 4m$, and one of the following holds:
\begin{itemize}
    \item [(1)] $R(T)/N \cong \mathbb{Z}_2 \times \mathbb{Z}_2$,  $i = 2$;
    \item [(2)] $R(T)/N \cong D_{2i}$,  $i \geq 3$ and $i \mid 2m$;
    \item [(3)] $R(T)/N \cong \mathbb{Z}_4 \times \mathbb{Z}_2$,  $i = 4$ and $i \nmid 2m$;
    \item [(4)] $R(T)/N \cong \mathrm{SD}_{2i}$,  $i \geq 3$, $i \neq 4$ and $i \nmid 2m$.
\end{itemize}

\end{theorem}

\begin{proof}
Since \(\Gamma\) is a \((G, 2)\)-distance-transitive graph, it is also locally \((G, 2)\)-distance-transitive, and hence Lemma \ref{DevillersGiudiciLiPraeger2012} applies.
Because   \(N\) is   intransitive on \(V(\Gamma)\),   using the \(G\)-arc-transitivity of \(\Gamma\), we know that each non-trivial \(N\)-orbit does not contain any edge of \(\Gamma\).
As \(N\) has at least 3 orbits in \(V(\Gamma)\), the fact \(\Gamma \not\cong \K_{x[y]}\) for any \(x \geq 3\) and \(y \geq 2\) implies that only  Lemma \ref{DevillersGiudiciLiPraeger2012} (iv) occurs.
Hence \(N\) is semiregular on the vertex set and \(\Gamma\) is a cover of \(\Gamma_N\).

By maximality of  \(N\), every  normal subgroup of the quotient group \(G/N\) is  transitive or has exactly two orbits on \(V(\Gamma_N)\), so \(G/N\) is quasiprimitive or bi-quasiprimitive on \(V(\Gamma_N)\).

Since \(\Gamma\) is \((G,2)\)-distance-transitive, we can easily show that \(\Gamma_N\) is \(G/N\)-arc-transitive.
If $\Gamma_N$ is a complete graph, then we are done.
In the remainder, we assume that \(\Gamma_N\) is a non-complete graph.

Take   two pairs of vertices \((C_1, C_3)\) and \((C'_1, C'_3)\) in \(\Gamma_N\) with
\(d_{\Gamma_N}(C_1, C_3) = d_{\Gamma_N}(C'_1, C'_3) = 2\).
Then as \(\Gamma\)  covers the quotient graph  \(\Gamma_N\), there exist \(c_i \in C_i\) and \(c'_i \in C'_i\) such that
\((c_1, c_3)\) and \((c'_1, c'_3)\) are two pairs of vertices of \(\Gamma\) with
\(d_{\Gamma}(c_1, c_3) = d_{\Gamma}(c'_1, c'_3) = 2\).
Since \(\Gamma\) is \((G, 2)\)-distance-transitive, there exists \(\alpha \in G\) such that \((c_1, c_3)^\alpha = (c'_1, c'_3)\).
Hence \((C_1, C_3)^\alpha = (C'_1, C'_3)\). In particular, \(\alpha\) induces an element of \(G/N\) that maps \((C_1, C_3)\) to \((C'_1, C'_3)\).
It follows that \(\Gamma_N\) is \((G/N, 2)\)-distance-transitive.
If \(\Gamma_N\) is \((G/N, 2)\)-arc-transitive, then by Lemma \ref{LiPan2008},  \(\Gamma\) is 2-arc-transitive, a contradiction.
Hence \(\Gamma_N\) is not \((G/N, 2)\)-arc-transitive.

Let \( H = \langle a \rangle \).
By Lemma \ref{N < T},
either $N\leq  R(H) < R(T)$ and \( \Gamma_N \) is  a  Cayley graph over \( R(T)/N \), or $ |N:N\cap R(H)|=2$  and $\Gamma_N$ a  circulant over the cyclic group \( R(H)/(N \cap R(H)) \) and $G$ contains a normal subgroup $X$ satisfying $|N:X|=2$ or $4$.

In the remainder, we assume that $\Gamma_N$ is non-complete.
Suppose for contradiction  $N \nleq R(H)$. Then  \( \Gamma_N \) is a  circulant over  the cyclic group \( R(H)/(N \cap H) \).
Since \( \Gamma_N \) is \( (G/N, 2) \)-distance-transitive but  not  \( (G/N, 2) \)-arc-transitive,  it follows from   \cite[Theorem 1]{ChenJinLi2019} that  \( \Gamma_N \) is  isomorphic to
$ \K_{m[b]}$ for some  $m\geq 3$ and $b\geq 2$, or   a Paley graph $P(p)$ for some  prime number $p$ satisfying $p\equiv 1 \pmod 4$.
However, by Lemma \ref{quot-notkmb-1},   $\Gamma_N\ncong \K_{m[b]}$ for any  $m\geq 3$ and $b\geq 2$, and by  \cite[Lemma 3.1]{HuangFengZhouYin2025}, $\Gamma_N\ncong P(p)$, a contradiction.

Thus   $N\leq  R(H) < R(T)$ and \( \Gamma_N \) is a  Cayley graph over \( R(T)/N \). Then \(\Gamma\) is a cyclic cover of  \(\Gamma_N\).
Moreover, \( N = \langle a^i \rangle \) where \( i \) is a divisor of \( 4m \) and \( i \neq 1, 4m \). By Lemma \ref{T/N}, one of the following holds:
 \begin{itemize}
    \item [(1)]\( R(T)/N \cong \mathbb{Z}_2 \times \mathbb{Z}_2 \), \( i = 2 \);
    \item [(2)]\( R(T)/N \cong {D}_{2i} \), \( i \geq 3 \) and $i \mid 2m$;
    \item [(3)]\( R(T)/N \cong \mathbb{Z}_4 \times \mathbb{Z}_2 \), \( i = 4 \) and $i \nmid 2m$;
    \item [(4)]\( R(T)/N \cong SD_{2i} \), \( i \geq 3 \), \( i \neq 4 \), and $i \nmid 2m$.
\end{itemize}

This completes   the proof.
\end{proof}

\begin{lemma}\label{SD2i}
Let $T = \mathrm{SD}_{8m}$ with $N \leq T \leq G$, and $N \unlhd G$. Suppose there exists an odd prime power $q = p^f$ such that $\operatorname{PSL}(2, q) \leq G/N \leq \operatorname{P}\Gamma L(2, q)$ and $2i = q+1$, where $i \geq 3$, $i \neq 4$, $i \nmid 2m$, and $i \mid 4m$. Then $T/N \not\cong \mathrm{SD}_{2i}$.

\end{lemma}
\begin{proof}
We proceed by contradiction that \(T/N \cong SD_{2i}\), where \(i \geq 3\), \(i \neq 4\), \(i \nmid 2m\). As $i \nmid 2m$ and $i \mid 4m$, we have $4 \mid i$. Thus \(8 \mid 2i = (q+1)\). Since \(\text{PSL}(2, q) \leq G/N \leq \text{P}\Gamma L(2, q)\) for some odd prime power \(q\), \( G/N = \mathrm{PSL}(2, q).o \) and \( o \leq \mathbb{Z}_2 \times \mathbb{Z}_f \). Note that $T/N = SD_{2i} = SD_{q+1} \leq G/N = \mathrm{PSL}(2, q).o$, by Theorem \ref{Dickson, 8.27}, it follows that $SD_{q+1} \nleq \mathrm{PSL}(2, q)$. Clearly, $SD_{q+1} \nleq o \leq \mathbb{Z}_2 \times \mathbb{Z}_f$.

Suppose first that $\mathbb{Z}_2 \nleq o$. Then $o \leq \mathbb{Z}_f$. Since $ SD_{q+1} \cap \operatorname{PSL}(2, q) \trianglelefteq SD_{q+1}$ and \[
SD_{q+1}/(SD_{q+1} \cap \operatorname{PSL}(2, q)) \cong \operatorname{PSL}(2, q) SD_{q+1}/\operatorname{PSL}(2, q) \leq o.\] This implies $SD_{q+1} \cap \operatorname{PSL}(2, q) \neq 1$ and $SD_{q+1}/(SD_{q+1} \cap \operatorname{PSL}(2, q))$ is a cyclic group. Applying Lemma \ref{normal group} and Lemma \ref{T/N}, it follows that $|SD_{q+1}/(SD_{q+1} \cap \operatorname{PSL}(2, q))| = 2$ or $4$. Thus $2 \mid |o|$ and $f$ is even. Since $q = p^f$ and $p$ is a odd prime, it follows that $q \equiv 1 \pmod{4}$. This contradicts the fact \(8 \mid (q+1)\).

Now suppose that $\mathbb{Z}_2 \leq o$. Then $SD_{q+1} \leq \mathrm{PSL}(2, q).o \leq \mathrm{PSL}(2, q).(\mathbb{Z}_2 \times \mathbb{Z}_f)$. Thus \[
SD_{q+1}/(SD_{q+1} \cap \operatorname{PSL}(2, q)) \cong \operatorname{PSL}(2, q) SD_{q+1}/\operatorname{PSL}(2, q) \leq o = \mathbb{Z}_2 \times \mathbb{Z}_f.\]

If $f$ is even, a similar argument yields a contradiction to $8 \mid (q+1)$.

We now turn to the case when $f$ is odd. Since $\mathbb{Z}_2 \times \mathbb{Z}_f$ is abelian group, applying Lemma \ref{normal group} and Lemma \ref{T/N}, it follows that $SD_{q+1}/(SD_{q+1} \cap \operatorname{PSL}(2, q))$ is isomorphic to one of the following groups: $\mathbb{Z}_2 \times \mathbb{Z}_2$, $\mathbb{Z}_4 \times \mathbb{Z}_2$, $\mathbb{Z}_4$, $\mathbb{Z}_2$. Given that $f$ is odd, it follows that $SD_{q+1}/(SD_{q+1} \cap \operatorname{PSL}(2, q))$ is isomorphic to $\mathbb{Z}_2$. Thus $SD_{q+1} \leq \text{P}GL(2, q)$. Applying \cite{CameronOmidiTayfehRezaie2006}, we have the group $\text{P}GL(2, q)$ contains no semidihedral subgroups. This is a contradiction.

We conclude the proof.
\end{proof}

\begin{lemma}\label{complete-1}
Let $\Gamma = \operatorname{Cay}(T, S)$ be a connected $(G,2)$-distance-transitive graph, where $T = SD_{8m}= \langle a, b \mid a^{4m} = 1,\, b^2 = 1,\, a^b = a^{2m-1} \rangle$ with $T \leq G \leq \operatorname{Aut}(\Gamma)$ and $m \geq 3$. Suppose $N$ is a normal subgroup of $G$ having at least three orbits such that $\Gamma$ is a cover of  $\Gamma_N$. If  $\Gamma_N$ is a complete graph, then $\Gamma$ is $(G,2)$-arc-transitive, and $\Gamma$ is isomorphic to one of the following graphs:

\begin{itemize}
    \item [(1)]    $ C_{8m}$.
    \item [(2)]    $  \K_{4m,4m} - 4m\K_2$.
   \item [(3)]   $ X_1(4,q)$, where $q \equiv 3 \pmod{4}$ and $q = 2m-1$ is a prime power.

    \end{itemize}

\end{lemma}
\begin{proof}
Suppose that  $\Gamma_N \cong \K_r$ with $r\geq 3$.
If $r=3$, then $\Gamma_N$ has valency 2. Since $\Gamma$ is a cover of  $\Gamma_N$, $\Gamma_N$ has valency 2, and so
$\Gamma\cong C_{8m}$, case (1) holds.
In the remainder, we assume that $r\geq 4$.

Let $H=\langle a \rangle$. Then   $H\cong \mathbb{Z}_{4m}$ has exactly two orbits on $V(\Gamma)$, and so $\Gamma$ is a bicircualnt.
Because     $\Gamma$  is a Cayley graph over $T = SD_{8m}$ with $m\geq 3$, $\Gamma$
has at least $24$ vertices.
Since $\Gamma = \operatorname{Cay}(T, S)$ is   $(G,2)$-distance-transitive, $\Gamma_N$ is $G/N$-arc-transitive, and so $G/N$ acts 2-transitively on $V(\Gamma_N)$.

If  \(\Gamma_N\) is  $(G/N, 2)$-arc-transitive, then   by Lemma \ref{LiPan2008},  $\Gamma$ is  $(G, 2)$-arc-transitive.
Since $\Gamma$ is a  bicirculant as a   cyclic cover of  $\K_r$,   it follows from   \cite[Lemma 4.1]{Jin2023} that  $\Gamma$ is isomorphic to either $\K_{4m,4m} - 4m\K_2$ or $X_1(4,q)$, where $q =r-1= 2m-1$ is a prime power and $q \equiv 3 \pmod{4}$.

Now suppose that $\Gamma_N$ is not $(G/N,2)$-arc-transitive. Then $G/N$ acts $2$-transitively but not $3$-transitively on $V(\Gamma_N)$.
Hence $G/N$ is one of the 2-transitive groups in \cite[Theorem 4.1]{Taylor1992} which are  not 3-transitive.

By Lemma \ref{N < T}, either
\begin{itemize}
    \item [(i)]   $N\leq  R(H) < R(T)$ and $\Gamma_N$ is a Cayley graph on $R(T)/N$; or
    \item [(ii)]   $ |N:N\cap R(H)|=2$ and $\Gamma_N$ is a circulant on  $R(H) N/N$. Furthermore,
    $G$ contains a normal subgroup $X$ satisfying $|N:X|=2$ or $4$.
    \end{itemize}

By Lemma \ref{|N|=p},   the covering transformation group $N$ can be taken as a cyclic group $\mathbb{Z}_p$ for some prime $p$.

Assume    $p=2$. If (ii) holds, then $ |N\cap R(H)|=1$ and $\Gamma$ is a standard double cover of  $\Gamma_N\cong \K_r$, and
 $\Gamma\cong \K_{4m,4m} - 4m\K_2$. Nevertheless, no arc-transitive subgroup of $\operatorname{Aut}(\K_{4m,4m} - 4m\K_2)$ that contains  a  regular subgroup  isomorphic to $D_{8m}$ or $SD_{8m}$, yielding a contradiction.

Assume    $p$ is an odd prime.
Then only   case (i) holds.
Since    $N = \langle a^i \rangle$, where $i$ divides $4m$ and $i \neq 1, 4m$, by Lemma \ref{T/N}, one of the following holds:
\begin{itemize}
    \item [(a)] $T/N \cong \mathbb{Z}_2 \times \mathbb{Z}_2$ and $i = 2$;
    \item [(b)] $T/N \cong D_{2i}$, where $i \geq 3$ and $i \mid 2m$;
    \item [(c)] $T/N \cong \mathbb{Z}_4 \times \mathbb{Z}_2$, where $i = 4$ and $i \nmid 2m$;
    \item [(d)] $T/N \cong SD_{2i}$, where $i \geq 3$, $i \neq 4$ and $i \nmid 2m$.
\end{itemize}

Moreover,     $G/N$ contains a cyclic subgroup $R(H)/N$ with two orbits on $V(\Gamma_N)$. Consequently, $G/N$ falls into one of the cases classified in \cite[Theorem 3.3]{Muller2013}.
We use Lemma  \ref{3-trans} to eliminate the $3$-transitive candidates, so one of the cases in \cite[Theorem 4.1]{Taylor1992} arises:

\begin{enumerate}
\item[(1)(b)] \(r = q + 1\) and \(\text{PSL}(2, q) \leq G/N \leq \text{P}\Gamma L(2, q)\) for some odd prime power \(q\geq 5\).
\item[(2)(a)] \(G/N \leq A\Gamma L(1, 16)\), the group of affine transformations of \(F_{16}\) extended by the automorphisms of \(F_{16}\).
\item[(2)(b)] \(SL(4, 2) \trianglelefteq (G/N)_0\), \(r = 16\).
\item[(2)(e)] \((G/N)_0 \cong A_6\) or \(A_7\), \(r = 16\).
\end{enumerate}

We first consider  cases: $(2)(a)$, $(2)(b)$ and $(2)(e)$. Then   \(Q := \text{soc}(G/N) \cong \mathbb{Z}_2^4\), \(\Gamma_N \cong \K_{r}\) with $r = 16$.

Since $r = 16$,  only cases $(b)$ and $(d)$ are possible.
Let $N \leq Y \trianglelefteq G$ such that $Y/N = Q$. Then $Y = N.Q$ acts transitively on $V(\Gamma)$.


Let $Y_2$ be a Sylow $2$-subgroup of $Y$. Then $Y_2 \cong \mathbb{Z}_2^4$ and $Y = N : Y_2 \cong \mathbb{Z}_p . \mathbb{Z}_2^4$. Set $C := C_Y(N)$. Then $C = N \times C_2$, where $C_2$ denotes a Sylow $2$-subgroup of $C$. By the $N/C$-theorem \cite[Satz 4.5]{Hup67}, $Y/C$ is isomorphic to a subgroup of $\operatorname{Aut}(N) \cong \mathbb{Z}_{p-1}$, so $Y/C$ is cyclic. Since $Y_2 \cong \mathbb{Z}_2^4$, we deduce that $8 \mid |C|$, and hence $C_2 \cong \mathbb{Z}_2^3$ or $\mathbb{Z}_2^4$.

As $Y \trianglelefteq G$, we have $C \trianglelefteq G$. Since $C_2$ is characteristic in $C$, it follows that $C_2 \trianglelefteq G$. Observe that $C_2$ has at least three orbits on $V(\Gamma)$. Then  $\Gamma$ is a cover of the normal quotient $\Gamma_{C_2}$, then  either     $C_2\leq  R(H) < R(T)$; or   $ |C_2:C_2\cap R(H)|=2$.
Since $C_2 \cong \mathbb{Z}_2^3$ or $\mathbb{Z}_2^4$, neither is possible.

It remains to consider  case $(1)(b)$. In this case, $r = q+1$ and $\operatorname{PSL}(2,q) \leq G/N \leq \operatorname{P}\Gamma L(2,q)$ for some odd prime power $q = p^f$.
By  Lemma \ref{SD2i},  case $(d)$  cannot occur.
It remains to show that the remaining three cases are impossible.

Suppose case $(a)$ occurs, so that $T/N \cong \mathbb{Z}_2\times\mathbb{Z}_2$. Then  $\Gamma_N \cong \K_4$, and   $\Gamma$ is a cyclic $(G,2)$-distance-transitive cover of $\K_4$. Hence $\Gamma$  is necessarily a cyclic cubic arc-transitive cover of $\K_4$. By \cite[Theorem 6.1]{FengKwak2007}, $\Gamma$ is either the $3$-cube $Q_3$ or the generalized Petersen graph $P(8,3)$. However, $Q_3$ has $8$ vertices and $P(8,3)$ has $16$ vertices. This contradicts the requirement that $\Gamma$ has at least $24$ vertices.

Next consider case $(b)$, where $T/N \cong D_{2i}$ with $i\geq 3$ and $i\mid 2m$. In this case, $G/N = \operatorname{PSL}(2,q).o$ contains the regular dihedral subgroup $T/N$. By \cite[Theorem 3.3]{SongLiZhang2014},  $r = q+1$ and  $q = e^f \equiv 3\pmod{4}$ for some prime $e$, and $o \leq \mathbb{Z}_2\times\mathbb{Z}_f$ does not contain the diagonal automorphism of $\operatorname{PSL}(2,q)$. Since  $N\cong \mathbb{Z}_p$ and  since $T/N \leq \operatorname{PSL}(2,q)$, we deduce that $T \cong D_{p(q+1)}$. This contradicts the fact that $T$ is a semi-dihedral group.

Finally, we deal with case $(c)$, in which $T/N \cong \mathbb{Z}_4\times\mathbb{Z}_2$. This yields $\Gamma_N \cong \K_8$ and $q=7$. Note that $\operatorname{P}\Gamma L(2,7)=\operatorname{PGL}(2,7)$, so  $\mathbb{Z}_4\times\mathbb{Z}_2 \leq G/N \leq \operatorname{PGL}(2,7)$. However, by \cite{CameronOmidiTayfehRezaie2006}, $\operatorname{PGL}(2,q)$ admits no subgroups isomorphic to $\mathbb{Z}_4\times\mathbb{Z}_2$. We thus reach a contradiction.


We conclude the proof.
\end{proof}

\begin{lemma}\label{bi-quasiprimitive}
Let $\Gamma$ be a connected $(G,2)$-distance-transitive Cayley graph on  a semi-dihedral group $T $  with   $R(T) \leq G \leq \operatorname{Aut}(\Gamma)$. If $G$ acts bi-quasiprimitively on $V(\Gamma)$, then $\Gamma$ is isomorphic to one of the following graphs: $\K_{4m,4m}$, $\K_{4m,4m}-4m\K_2$, $B(\mathrm{PG}(d-1,q))$, or $B'(\mathrm{PG}(d-1,q))$, where $d \geq 3$, $q$ is a prime power, and $\frac{q^d - 1}{q - 1} = 4m$. Moreover, $\Gamma$ is $(G, 2)$-arc-transitive.
\end{lemma}

\begin{proof}
Suppose that    $G$ is bi-quasiprimitive on $V(\Gamma)$. By the definition of bi-quasiprimitive group actions,  $G$ has a minimal normal subgroup $M$ with  exactly two orbits on $V(\Gamma)$. Since  $\Gamma$ is $G$-arc-transitive, each $M$-orbit does not contain any edge of $\Gamma$. This implies that  $\Gamma$ is bipartite, with its two bipartite halves precisely being the two $M$-orbits of $V(\Gamma)$. Denote these two bipartite parts by $\Delta_0$ and $\Delta_1$.

As $\Gamma$ is a Cayley graph over $T$, the right regular representation  $R(T)$ acts regularly on $V(\Gamma)$. Recall that the semi-dihedral group of order $8m$ has the standard representation    $T = \langle a, b \mid a^{4m} = b^2 = 1, a^b = a^{2m-1} \rangle$ where   $m \geq 3$.  Let $H = \langle a \rangle \cong \mathbb{Z}_{4m}$. Then $R(H)$ is a subgroup of $R(T)$ of index $2$ with   exactly two orbits on $V(\Gamma)$. Denote these two $R(H)$-orbits by $H_0$ and $H_1$. In particular, $\Gamma$ is a bicirculant relative to   $H$.

If the two orbits $H_0$ and $H_1$ of the cyclic subgroup $H$ coincide with the bipartite halves $\Delta_0$ and $\Delta_1$ of $\Gamma$, then $\Gamma$ falls into   the classification family established  in \cite[Proposition 5.1]{DevillersGiudiciJin2022}. Otherwise, if $H_0$ and $H_1$ differ from the bipartite halves $\Delta_0$ and $\Delta_1$, then $\Gamma$ is determined by \cite[Proposition 5.2]{DevillersGiudiciJin2022}. Combining these two cases, $\Gamma$ is isomorphic to one of the following graphs: $\K_{4m,4m}$, $\K_{4m,4m}-4m\K_2$, $B(H(11))$, $B'(H(11))$, $G(2p,r)$, $B(\mathrm{PG}(d-1,q))$, or $B'(\mathrm{PG}(d-1,q))$, where $d \geq 3$, $q$ is a prime power, and $\frac{q^d - 1}{q - 1} = 4m$.

Note that the order of the semi-dihedral group $|T|=8m$ is divisible by $8$, so the graphs $B(H(11))$, $B'(H(11))$ and $G(2p,r)$ can be excluded. This leaves only the following candidates for $\Gamma$: $\K_{4m,4m}$, $\K_{4m,4m}-4m\K_2$, $B(\mathrm{PG}(d-1,q))$ and $B'(\mathrm{PG}(d-1,q))$, where $d \geq 3$, $q$ is a prime power and $\frac{q^d - 1}{q - 1} = 4m$.

It remains to verify the $(G,2)$-arc-transitivity of all candidate graphs.
For the graph $\K_{4m,4m} - 4m\K_2$,  its  $(G,2)$-arc-transitivity is guaranteed by \cite[Lemma 3.6]{JinLiTan2025}. For the complete bipartite graph $\K_{4m,4m}$, it follows from  \cite[Lemma 3.3]{CorrJinSchneider2017} that  $(G,2)$-distance-transitivity  implies $(G,2)$-arc-transitivity. If $\Gamma$ is isomorphic to $B(\operatorname{PG}(d-1,q))$ or $ B'(\operatorname{PG}(d-1,q))$,  then $\operatorname{PGL}(d,q) \leq G^+ = G_{\Delta_0} = G_{\Delta_1} \leq \operatorname{P\Gamma L}(d,q)$, which means that $\Gamma$ is $(G,2)$-arc-transitive.
\end{proof}

\bigskip

We now establish  our main theorem, which gives a complete classification of  connected $(G, 2)$-distance-transitive Cayley graph on   semi-dihedral groups.
\bigskip

	\begin{proof}[\bf{Proof of Theorem \ref{thm1.1}}.]

Let \(\Gamma\) be a connected \((G, 2)\)-distance-transitive Cayley graph on  the semi-dihedral group $T = SD_{8m} = \langle a, b \mid a^{4m} = 1, \, b^2 = 1, \, a^b = a^{2m-1} \rangle$  with $m \geq 3$ and \( T \leq G \leq \text{Aut}(\Gamma) \). Then \(\Gamma\) has at least $24$ vertices.
If $\Gamma$ has valency 2, then it is isomorphic to $C_{8m}$. In the remainder, we assume that
$\Gamma$ has valency at least 3.

Let $H = \langle a \rangle \cong \mathbb{Z}_{4m}$. Then $R(H)$ is an index-2  subgroup of $R(T)$ and partitions  $V(\Gamma)$  into exactly two orbits, denoted by   $H_0$ and $H_1$. Consequently, $\Gamma$ is a bicirculant relative to   $H$.

Suppose that \( G \) acts quasiprimitively on \( V(\Gamma) \). Since $\Gamma$ is  $G$-arc-transitive and  $H\cong \mathbb{Z}_{4m}$, it follows from Lemma \ref{DevillersGiudiciJin2022} that   $\Gamma$ is isomorphic to either $\K_{4m[2]}$ or $\K_{4m,4m} - 4m\K_2$.

It remains to consider the case that $G$ is not quasiprimitive on $V(\Gamma)$. Under this assumption $G$ has at least one intransitive normal subgroup, and this subgroup has at least $2$ orbits on $V(\Gamma)$.

If every nontrivial normal subgroup of \(G\) has at most two orbits on \(V(\Gamma)\), and if there exists a nontrivial normal subgroup of \(G\) with exactly two orbits, then \(G\) acts bi-quasiprimitively on \(V(\Gamma)\) and \(\Gamma\) is bipartite. By Lemma \ref{bi-quasiprimitive}, \(\Gamma\) is isomorphic to one of the following graphs: \(\mathrm{K}_{4m,4m}\), \(\mathrm{K}_{4m,4m}-4m\mathrm{K}_2\), \(B(\mathrm{PG}(d-1,q))\), and \(B'(\mathrm{PG}(d-1,q))\), where \(d\ge 3\), \(q\) is a prime power, and \(\frac{q^d-1}{q-1}=4m\).

Let \(N\) be a normal subgroup of \(G\) that is maximal among all normal subgroups with at least three orbits on \(V(\Gamma)\). As \(\Gamma\) is \((G,2)\)-distance-transitive, it is locally \((G,2)\)-distance-transitive. Applying Lemma \ref{DevillersGiudiciLiPraeger2012}, either \(\Gamma\cong \mathrm{K}_{x[y]}\) for integers \(x\ge3,y\ge2\) satisfying \(xy=8m\), or \(\Gamma\) is a cover  of \(\Gamma_N\). The first case has already been settled.

We now assume that \(\Gamma \not\cong \mathrm{K}_{x[y]}\) for all integers \(x\ge 3\) and \(y\ge 2\), so that \(\Gamma\) is a cover of \(\Gamma_N\). By Theorem \ref{(G/N, s)-distance-transitive}, \(\Gamma_N\) is a quotient graph that is either complete or a noncomplete \((G/N,2)\)-distance-transitive graph. Furthermore, the quotient group \(G/N\) acts faithfully on \(V(\Gamma_N)\) and is either quasiprimitive or bi-quasiprimitive on this vertex set.

If \( \Gamma_N \) is isomorphic to a complete graph, then by Lemma \ref{complete-1}, $\Gamma$ is isomorphic to the graph \(\K_{4m, 4m} - 4m\K_2\) or \(X_1(4, q)\) where \(q \equiv 3 \pmod{4}\) and \(q = 2m - 1\).

Now suppose that \( \Gamma_N \) is a \((G/N, 2)\)-distance-transitive
non-complete graph,  and \(G/N\) is faithful and either is quasiprimitive or bi-quasiprimitive on \(V(\Gamma_N)\).

First assume that \(N = \langle a^i \rangle \)  where \(i\) is a divisor of \(4m\) and \(i \neq 1, 4m\). By Theorem \ref{(G/N, s)-distance-transitive}, one of the following four cases occurs for the quotient group \(R(T)/N\):
 \begin{itemize}
    \item [(1)]\( (R(T)/N \cong \mathbb{Z}_2 \times \mathbb{Z}_2 \), \( i = 2 \);
    \item [(2)]\( (R(T)/N \cong {D}_{2i} \), \( i \geq 3 \) and $i \mid 2m$;
    \item [(3)]\( (R(T)/N \cong \mathbb{Z}_4 \times \mathbb{Z}_2 \), \( i = 4 \) and $i \nmid 2m$;
    \item [(4)]\( (R(T)/N \cong SD_{2i} \), \( i \geq 3 \), \( i \neq 4 \), and $i \nmid 2m$.
\end{itemize}

If   Case $(1)$ holds, then  the quotient graph \(\Gamma_N\) has exactly four vertices, so it is isomorphic to either \(\mathrm{K}_4\) or \(C_4\),  contradicting that \(\Gamma_N\) is non-complete with valency at least 3.

We now turn to Case $(2)$. Then   \(\Gamma_N\) is a \((G/N,2)\)-distance-transitive Cayley graph over a dihedral group,
it is in particular   2-distance-transitive. By   \cite[Theorem 1.1]{HuangFengZhouYin2025},  \(\Gamma_N\) is   either \(2\)-arc-transitive or a complete multipartite graph of the form \(\mathrm{K}_{x'[y']}\).
Recall that \(\Gamma\) itself is not isomorphic to any complete multipartite graph; together with Lemma \ref{quot-notkmb-1}, this excludes the possibility \(\Gamma_N\cong\mathrm{K}_{x'[y']}\). Consequently, \(\Gamma_N\) must be \(2\)-arc-transitive, and thus falls into the classification of 2-arc-transitive dihedrants given in  Lemma \ref{Du2008}.
Since   \(G/N\) acts either quasiprimitively or bi-quasiprimitively on \(V(\Gamma_N)\), it follows that  \(\Gamma_N\) is isomorphic to one of the following graphs:

\begin{itemize}
\item [(a)] \( \K_{i,i} \);
\item [(b)] \( B(H_{11}) \) or \( B^{\prime}(H_{11}) \);
\item [(c)] \( B(\operatorname{PG}(d,q)) \) or \( B'(\operatorname{PG}(d,q)) \), where \( i = \frac{q^d - 1}{q - 1} \), \( d \geq 2 \), and \( q \) is a prime power.
\end{itemize}

As \(\Gamma_N\) is \((G/N, 2)\)-arc-transitive,? it follows from Lemma \ref{LiPan2008} that \(\Gamma\) is 2-arc-transitive.

If  \(\Gamma_N\cong \mathrm{K}_{i,i}\), then  by \cite[Lemma 4.3]{Jin2023},  \(i=4\) and \(\Gamma\cong X(2,2)\),   contradicting \(|V(\Gamma)|=8m\ge 24\). Hence case (a) is not possible.

By  \cite[Lemma 4.4]{Jin2023}, no \(2\)-arc-transitive bicirculant arises as a regular cyclic cover of either of \( B(H_{11})\) or \(B'(H_{11})\), so  case (b) cannot occur.

 According to \cite[Lemma 4.6(1)]{Jin2023}, \(B(\operatorname{PG}(d, q))\) admits no   2-arc-transitive regular cyclic bicirculant covers  for   \(d \geq 2\) and \(q\)  a prime power. Hence \(\Gamma_N\ncong  B(\operatorname{PG}(d, q))\).

Finally, suppose \(\Gamma_N\cong B'(\operatorname{PG}(d,q))\) with \(d\ge 2\), \(q\) a prime power and \(\frac{q^d-1}{q-1}=i\). Then by   \cite[Lemma 4.6(2)]{Jin2023}, \(\Gamma\) is isomorphic  to either \(X'(3,2)\) or  \(\Gamma(d,q,r)\) for some divisor \(r\mid q-1\). Since  \(|V(\Gamma)|=8m\neq |V(X'(3,2))|=28\), the case \(\Gamma\cong X'(3,2)\) is impossible. Therefore \(\Gamma\cong \Gamma(d,q,r)\), where \(d\ge2\), \(q\) is a prime power, \(r\mid q-1\), and the vertex count satisfies $|V(\Gamma(d,q,r))|=\frac{r}{q-1}\cdot |V(\Gamma(d,q))|=8m$.

We now proceed to Case $(3)$, in which \(T/N \cong \mathbb Z_4\times\mathbb Z_2\), \(i=4\) and \(m\) is odd. As \(\Gamma_N\) is a \((G/N,2)\)-distance-transitive Cayley graph of \(T/N\cong\mathbb Z_4\times\mathbb Z_2\), it follows that \(|V(\Gamma_N)|=8\).

Assume first that \(G/N\) is quasiprimitive on \(V(\Gamma_N)\). By Lemma \ref{B-group}, \(G/N\) is primitive on \(V(\Gamma_N)\). In view of Lemma \ref{thm:25.5}, the abelian group \(\mathbb{Z}_4\times\mathbb{Z}_2\) is a \(B\)-group, which implies that \(G/N\) acts \(2\)-transitively on \(V(\Gamma_N)\). Consequently, \(\Gamma_N\cong \mathrm{K}_8\). However, \(\Gamma_N\) is a noncomplete \((G/N,2)\)-distance-transitive graph, so the isomorphism \(\Gamma_N\cong \mathrm{K}_8\) is impossible.

Assume now that \(G/N\) is bi-quasiprimitive on \(V(\Gamma_N)\). By Lemma \ref{bi-quasiprimitive}, \(\Gamma_N\) is isomorphic to either \(\mathrm{K}_{4,4}\) or \(\mathrm{K}_{4,4}-4\mathrm{K}_2\). Since \(\Gamma_N\) is \((G/N,2)\)-arc-transitive, Lemma \ref{LiPan2008} implies that \(\Gamma\) is \((G,2)\)-arc-transitive. If \(\Gamma_N\cong \mathrm{K}_{4,4}\), then \cite[Lemma 4.3]{Jin2023} yields \(\Gamma\cong X(2,2)\), contradicting the fact that \(|V(\Gamma)|=8m\ge 24\). If alternatively \(\Gamma_N\cong \mathrm{K}_{4,4}-4\mathrm{K}_2\), then \(\Gamma\) is a \((G,2)\)-arc-transitive bicirculant that arises as a regular cyclic cover of \(\mathrm{K}_{4,4}-4\mathrm{K}_2\) with \(|N|=m\) and \(m\ge 3\). Applying \cite[Lemma 4.7]{Jin2023}, \(\Gamma\) is isomorphic to either \(\operatorname{GP}(12,5)\) or \(\operatorname{GP}(24,5)\). Since \(m\) is odd, \(\Gamma\) cannot be \(\operatorname{GP}(24,5)\). Furthermore, by \cite[Theorems 1--2, pp. 217--218]{FruchtGraverWatkins1971}, the automorphism group of \(\operatorname{GP}(12,5)\) is isomorphic to \(S_4\times S_3\). As this group contains no subgroup isomorphic to the semi-dihedral group \(T=SD_{8m}\), the graph \(\Gamma\) cannot be \(\operatorname{GP}(12,5)\) either.

Consider Case $(4)$, where \(T/N \cong SD_{2i}\) with \(i\ge 3\), \(i\neq 4\) and \(i\nmid 2m\). In particular, \(4\mid i\).

Assume first that \(G/N\) is quasiprimitive on \(V(\Gamma_N)\). Since \(\Gamma_N\) is a connected \((G/N,2)\)-distance-transitive Cayley graph of a semi-dihedral group, it is a connected \(G/N\)-arc-transitive bicirculant with respect to the cyclic subgroup of order \(i\). Invoking Lemma \ref{DevillersGiudiciJin2022}, \(\Gamma_N\) is isomorphic to the graph \(\mathrm{K}_{2i}\), \(\mathrm{K}_{i[2]}\), or \(\mathrm{K}_{i,i}-i\mathrm{K}_2\) where \(\operatorname{PGL}(d,q)\le G/N\le \operatorname{P\Gamma L}(d,q)\) with \(i=(q^d-1)/(q-1)\). Recall that \(\Gamma\) is not a complete multipartite graph; together with Lemma \ref{quot-notkmb-1}, this excludes the case \(\Gamma_N\cong \mathrm{K}_{i[2]}\). Moreover, as \(\Gamma_N\) is a noncomplete \((G/N,2)\)-distance-transitive graph, the complete graph \(\mathrm{K}_{2i}\) is also ruled out. Consequently, \(\Gamma_N\cong \mathrm{K}_{i,i}-i\mathrm{K}_2\), and there holds \(\operatorname{PGL}(d,q)\le G/N\le \operatorname{P\Gamma L}(d,q)\) with \(i=(q^d-1)/(q-1)\).
Recall that the semi-dihedral group admits the presentation \(SD_{2i}=\langle x,b\mid x^i=b^2=1, x^b=x^{\frac{i}{2}-1}\rangle\), and \(T/N \le SD_{2i}\le G/N\le \operatorname{P\Gamma L}(d,q)\). In particular, \(\langle x\rangle\le \operatorname{P\Gamma L}(d,q)\). Since \(\langle x\rangle\) is normal in \(SD_{2i}\), we obtain \(SD_{2i}\le N_{\operatorname{P\Gamma L}(d,q)}(\langle x\rangle)\). By \cite[p. 187, Satz 7.3]{Hup67}, the normalizer \(N_{\operatorname{P\Gamma L}(d,q)}(\langle x\rangle)\) takes the form \(\langle x\rangle:\langle \delta\rangle\), where \(\langle \delta\rangle\cong\mathbb{Z}_d\) and \(x^\delta=x^q\). As \(x^b=x^{\frac{i}{2}-1}\) and \(\langle x,b\rangle\le \langle x,\delta\rangle\), there exists an integer \(r\) with \(0\le r<d\) such that \(x^{\delta^r}=x^{q^r}=x^b=x^{\frac{i}{2}-1}\), which yields the congruence \(q^r\equiv \frac{i}{2}-1\pmod{i}\).
Furthermore, substituting \(i=(q^d-1)/(q-1)\) gives \(2q^r = i-2 = q^{d-1}+q^{d-2}+\cdots+q-1\). A straightforward inspection shows that this equation holds only if \(r=0\) and \(q=3\). On the other hand, since \(SD_{2i}\cap \operatorname{PGL}(d,q)\) is normal in \(SD_{2i}\), the quotient \(SD_{2i}/(SD_{2i}\cap \operatorname{PGL}(d,q))\) is isomorphic to a subgroup of \(\mathbb{Z}_f\), where \(q=p^f\) for some prime \(p\). Hence this quotient group is cyclic. Applying Lemma \ref{normal group} and Lemma \ref{T/N}, the order of this cyclic quotient is either \(2\) or \(4\), forcing \(f\) to be even. Since $q = p^f$ for some prime $p$ and $q = 3$, we have $f = 1$. This yields a final contradiction.

Now assume that \(G/N\) is bi-quasiprimitive on \(V(\Gamma_N)\). Since \(\Gamma\) is a cover of \(\Gamma_N\) and the quotient graph \(\Gamma_N\) is \((G/N,2)\)-distance-transitive, by Lemma \ref{bi-quasiprimitive}, \(\Gamma_N\) is \((G/N,2)\)-arc-transitive and is isomorphic to one of the following graphs: \(\mathrm{K}_{i,i}\), \(\mathrm{K}_{i,i}-i\mathrm{K}_2\), \(B(\mathrm{PG}(d-1,q))\), or \(B'(\mathrm{PG}(d-1,q))\), where \(d\ge 3\), \(q\) is a prime power, and \(i=(q^d-1)/(q-1)\). Furthermore, Lemma \ref{LiPan2008} implies that \(\Gamma_N\) is \((G,2)\)-arc-transitive.

If \(\Gamma_N\) is isomorphic to \(K_{i,i}\), then by \cite[Lemma 4.3]{Jin2023}, \(i = 4\) and \(\Gamma\) is the graph \(X(2, 2)\), which contradicts the fact that \(|V(\Gamma)| = 8m \geq 24\).

We now consider the case \(\Gamma_N \cong \mathrm{K}_{i,i} - i\mathrm{K}_2\), and let \(|N| = d = \frac{4m}{i}\). Since the covering relation preserves vertex valency and \(\Gamma\) has valency at least \(3\), the quotient graph \(\Gamma_N\) is also of valency at least \(3\), which implies \(i \ge 5\). Recall that \(T/N \cong SD_{2i}\), so \(4 \mid i\). Moreover, the divisibility conditions \(i \mid 4m\) and \(i \nmid 2m\) together force \(d = 4m/i\) to be odd. Applying \cite[Lemma 4.7]{Jin2023}, we conclude that \(\Gamma \cong K_{q+1}^{2d}\), where \(q\) is an odd prime power, \(d \ge 2\) divides \(q-1\), \(4m = d(q+1)\), and \(i = q+1 \ge 6\).

Moreover, by \cite[Lemma 4.6(1)]{Jin2023}, \(\Gamma_N\) cannot be isomorphic to \(B(\operatorname{PG}(d-1,q))\) for any integer \(d\ge 3\) and any prime power \(q\).

Finally, suppose \(\Gamma_N \cong B'(\operatorname{PG}(d-1,q))\) with \(d\ge 3\), \(q\) a prime power, and \(i = \frac{q^d-1}{q-1}\). According to \cite[Lemma 4.6(2)]{Jin2023}, \(\Gamma\) is isomorphic either to \(X'(3,2)\) or to \(\Gamma(d,q,r)\) for some divisor \(r \mid q-1\). As in the previous analysis, the vertex cardinality of \(X'(3,2)\) contradicts the vertex number condition of \(\Gamma\), so the case \(\Gamma \cong X'(3,2)\) is impossible. Consequently, we obtain \(\Gamma \cong \Gamma(d,q,r)\), where \(r \mid q-1\) with $|V(\Gamma(d, q, r))| = \frac{r}{q-1} \cdot |V(\Gamma(d, q))| = 8m$.

In the remainder of this proof, we assume that \(N \nleq \langle a \rangle\). By Theorem \ref{(G/N, s)-distance-transitive}, \(N\) is a proper subgroup of \(T\) and \(T/N \cong \mathbb{Z}_4\). It follows that \(\Gamma_N\) is a four-vertex graph, so \(\Gamma_N\) is isomorphic to either \(\mathrm{K}_4\) or \(C_4\). Since \(\Gamma\) is a metacyclic cover of  \(\Gamma_N\) and \(\Gamma\) has valency at least \(3\), the valency of \(\Gamma_N\) is also at least \(3\), which rules out the cycle graph \(C_4\). Furthermore, \(\Gamma_N\) is a noncomplete \((G/N,2)\)-distance-transitive graph, so the complete graph \(\mathrm{K}_4\) is also excluded. Thus this case is impossible.

This concludes the proof.
    \end{proof}
   \medskip

\begin{lemma}
    The graph $X_1(4,q)$ where \(q \equiv 3 \pmod{4}\) and \(q = 2m - 1\) is a  Cayley graph of the group $T= \mathrm{SD}_{8m} = \langle a, b \mid a^{4m} = b^2 = 1,\ a^b = a^{2m-1} \rangle,\ m \geq 3$.
\end{lemma}

\begin{proof}
By their definitions, it may be easily seen that $K_{q+1}^{4}$ is isomorphic to $X_{1}(4,q)$ when $q \equiv 3 \pmod{4}$.
Let $K$ denote the covering transformation group of $K_{q+1}^4$, regarded as a $4$-fold regular cover of the base graph $K_{q+1}$.
According to \cite[Theorem~5.3]{Marusic2003}, $K$ is normal in $\operatorname{Aut} K_{q+1}^4$ and
\[
\operatorname{Aut} K_{q+1}^4 / K \cong \operatorname{P\Gamma L}(2,q),\qquad K \cong \mathbb{Z}_4.
\]

Recall that $K_{q+1}^4$ is a regular $\mathbb{Z}_2$-cover of the graph $Y \cong K_{q+1,q+1} - (q+1)K_2$,
whose vertex set admits a bipartition $V(Y) = V \cup V'$, with $V$ and $V'$ each isomorphic to the projective line $\operatorname{PG}(1,q)$.
Let $\sigma \in \operatorname{P\Gamma L}(2,q)$ be a Singer cycle of order $q+1$.
Its natural action on $Y$ stabilizes both $V$ and $V'$, each forming a single orbit of size $q+1$.
Now take a lift $\tilde{\sigma}$ of $\sigma$ to the covering graph $K_{q+1}^4$.
The index-$2$ subgroup of the full automorphism group of $K_{q+1}^4$ that preserves the bipartition (and does not interchange its two vertex parts) contains a lift of $\operatorname{P\Gamma L}(2,q)$, and consequently a lift of the simple group $\operatorname{PSL}(2,q)$.
As $\operatorname{PSL}(2,q)$ acts faithfully on each bipartite component, its lifted counterpart also acts faithfully on each part.
This restricts the orbit structure of $\tilde{\sigma}$ on $V(K_{q+1}^4)$ to exactly two cases.
Hence a lift $\tilde{\sigma}$ of $\sigma$ has either two orbits of length $2(q+1)$ or four orbits of length $q+1$.
If a lift $\tilde{\sigma}$ of $\sigma$ had four orbits of length $q+1$, according to \cite[Theorem~5.3]{Marusic2003}, this is impossible.

We next consider the graph having an automorphism with two orbits of equal lengths $2(q+1)$, then $\langle \sigma \rangle$ lifts to a cyclic group $\langle \tilde{\sigma} \rangle$.
Since $\langle {\sigma} \rangle \cong \mathbb{Z}_{q+1} = \mathbb{Z}_{2m}$, it follows that $\langle \tilde{\sigma} \rangle \cong \mathbb{Z}_{2q+2} = \mathbb{Z}_{4m}$.
Therefore, the group $\operatorname{Aut} K_{q+1}^4$ is obtained via a non-split extension of $K$ by $\operatorname{P\Gamma L}(2,q)$.
Let $ \tilde{x}, \tilde{y}$ be two vertices in $V(K_{q+1}^4)$.
Then we have $\tilde{x} \cdot \tilde{y} = \widetilde{xy} \, f(x,y)$, for a binary function $f\colon \operatorname{P\Gamma L}(2,q)\times \operatorname{P\Gamma L}(2,q)\to K$.
Hence it is easy to see that $(\tilde{\sigma}^2) = \widetilde{(\sigma)^2} f(\sigma, \sigma)$.
Then
\[
\tilde{\sigma}^{2m} = \widetilde{\sigma^{2m}} \prod_{i=1}^{2m-1} f(\sigma^i, \sigma) = \tilde{1}\prod_{i=1}^{2m-1} f(\sigma^i, \sigma) = z .
\]
Thus $z$ is an element of order $2$ in $K$.

By \cite[p.~187, Satz~7.3]{Hup67}, the normalizer $N_{\operatorname{P\Gamma L}(2,q)}(\langle \sigma\rangle)$ takes the form $\langle \sigma\rangle: \mathbb{Z}_2$.
Thus there exists an involution $\tau \in \operatorname{P\Gamma L}(2,q)$ such that $\tau \sigma \tau = \sigma^{q} = \sigma^{2m-1} =\sigma^{-1}$.
Assume that $\tau$ has a lift $\tilde{\tau}$ in $\operatorname{Aut} K_{q+1}^4$.
We may choose the representative $\tilde{\tau}$ such that
\[
f(\tau,\tau)=1,\quad f(\tau,\sigma)=1,\quad f(\tau\sigma,\tau\sigma)=z .
\]
Since $f(\tau, \tau) = 1$, we have $(\tilde{\tau})^2 = \widetilde{\tau^2}f(\tau, \tau) = \tilde{1}$.
In addition, due to $f(\tau, \sigma) = 1$ and $f(\tau \sigma, \tau \sigma) = z$, it follows that
\[
\tilde{\tau} \tilde{\sigma} \tilde{\tau} \tilde{\sigma} = (\tilde{\tau} \tilde{\sigma})^2 = (\widetilde{\tau\sigma} f(\tau, \sigma))^2 = (\widetilde{\tau\sigma})^2
= (\widetilde{\tau\sigma\tau\sigma}) (f(\tau \sigma, \tau \sigma)) = \tilde{1} \, z = z .
\]
Then we have
\[
\tilde{\tau} \tilde{\sigma} \tilde{\tau} = z \tilde{\sigma}^{-1} = \tilde{\sigma}^{2m}\tilde{\sigma}^{-1} = \tilde{\sigma}^{2m-1}.
\]
Let $a = \tilde{\sigma}$ and $b = \tilde{\tau}$. Then
\[
T = \mathrm{SD}_{8m} = \langle a, b \mid a^{4m} = b^2 = 1,\ a^b = a^{2m-1} \rangle .
\]

It remains to prove that $T$ acts transitively (and hence regularly) on $V(X_1(4,q))$, which has $8m$ vertices.
From the orbit analysis we know that $a$ has exactly two orbits of length $4m$; denote them by $O_1$ and $O_2$.
The element $z = a^{2m}$ is the unique central involution in $\langle a \rangle$; it has no fixed vertices and must interchange $O_1$ and $O_2$, because $a$ acts as a $4m$-cycle on each orbit and its $2m$-th power sends every vertex to the other vertex in the same fibre of the $\mathbb{Z}_2$-cover.
Using the relation $a^b = a^{2m-1} = a^{-1}z$, we obtain
\[
b^{-1} a b = a^{-1} z. \tag{1}
\]
Suppose, for contradiction, that $b$ preserves the orbits, i.e., $O_1^b = O_1$ and $O_2^b = O_2$.
Take an arbitrary vertex $x \in O_1$. Applying (1) we get
\[
x^{b^{-1} a b} = x^{a^{-1} z}.
\]
The left-hand side: $x^{b^{-1}} \in O_1$, then $x^{b^{-1} a} \in O_1$, finally $x^{b^{-1} a b} \in O_1$, so it belongs to $O_1$.
The right-hand side: $x^{a^{-1}} \in O_1$, and then multiplication by $z$ sends it into $O_2$ (since $z$ swaps the two orbits). Thus the right-hand side lies in $O_2$.
This contradicts the fact that $O_1 \cap O_2 = \varnothing$. Therefore $b$ must interchange $O_1$ and $O_2$.
Since $\langle a \rangle$ is transitive on each of $O_1$ and $O_2$, the group $\langle a, b \rangle = T$ is transitive on the whole vertex set $O_1 \cup O_2$.
Because $|T| = 8m$ equals the number of vertices, the action is regular.
By Sabidussi's theorem, the existence of a regular subgroup $T \le \operatorname{Aut}(X_1(4,q))$ implies that $X_1(4,q)$ is a Cayley graph of $T$.
\end{proof}



\end{document}